%% file: RB_EWP.tex
\documentclass{article}

\usepackage{mathptmx}       
\usepackage{helvet}         
\usepackage{courier}        
\usepackage{type1cm}

\usepackage{graphicx}

\usepackage{multicol}       
\usepackage[bottom]{footmisc}

\usepackage{amssymb}
\usepackage{amsmath}
\usepackage{amsthm}
\usepackage{mathtools}

\usepackage[numbers,sort&compress]{natbib}

\usepackage{algorithm}
\usepackage{algpseudocode}

\usepackage{color}

\def\RR{\mathbb{R}}

\DeclareMathOperator*{\argmax}{arg\,max}

\def\muPDE{$\mu$\hspace{-0.3mm}PDE}
\def\muEVP{$\mu$\hspace{-0.3mm}EVP}

\def\anzEW{K}

\def\TOL{\mathrm{\varepsilon_\mathrm{tol}}}
\def\epslambda{\mathrm{\varepsilon_\mathrm{\lambda}}}
\def\epsproj{\mathrm{\varepsilon_\mathrm{proj}}}

\def\muRef{{\hat{\mu}}}
\def\aRef{{\hat{a}}}
\def\eRef{{\hat{e}}}
\def\ARef{{\hat{A}}}

\def\aNorm#1{\left\|#1\right\|_{\mu;V}}
\def\aNormRef#1{\left\|#1\right\|_{\muRef; V}}
\def\aNormDual#1{\left\|#1\right\|_{\mu; V'}}
\def\aNormRefDual#1{\left\|#1\right\|_{\muRef; V'}}

\def\Kcal{\mathcal K}

\def\lambdaRB#1{\lambda_{\mathrm{red},\,#1}}
\def\lambdaRBnoindex{\lambda_{\mathrm{red}}}
\def\uRB#1{u_{\mathrm{red},\,#1}}
\def\uRBnoindex{u_{\mathrm{red}}}
\def\KcalRB#1{\mathcal K_{\mathrm{red},\,#1}}

\def\dtilde{\tilde{d}}
\def\d{d}

\newtheorem{thrm}{Theorem}[section]
\newtheorem{rmrk}[thrm]{Remark}
\newtheorem{crllr}[thrm]{Corollary}

\makeindex
\begin{document}
\title{ Simultaneous Reduced Basis Approximation\\of Parameterized Elliptic Eigenvalue Problems
\thanks{We would like to gratefully acknowledge the funds provided by the “Deutsche Forschungsgemeinschaft” under the contract/grant numbers: WO-671/13-1.}}
\date{}
\author{
{\sc
 Thomas Horger \thanks{Corresponding author. Email: horger@ma.tum.de}, Barbara Wohlmuth \thanks{Email: wohlmuth@ma.tum.de} ,Thomas Dickopf\thanks{Email: dickopf@ma.tum.de}
 }\\[2pt]
  M2 -- Zentrum Mathematik, Technische Universit\"at M\"unchen,\\ Boltzmannstra\ss{}e 3, 85748 Garching, Germany\\
}
\maketitle

\begin{abstract}
The focus is on a model reduction framework for parameterized elliptic eigenvalue problems
by a reduced basis method.
In contrast to
the standard single output case,
one is interested in approximating several outputs
simultaneously,
namely a certain number of the smallest eigenvalues.
For a fast and reliable evaluation of these input-output relations,
we analyze a~posteriori error estimators for eigenvalues. 
Moreover,
we present different greedy strategies and study systematically their performance.
Special attention needs to be paid to multiple eigenvalues
whose appearance is parameter-dependent.
Our methods are of particular interest for applications in vibro-acoustics.
\end{abstract}

\input{intro}

\input{problem}

\input{errorestimator}

\input{algorithms}

\input{numerics}

\input{conclusion}
\input{references}

\end{document}

%% file: intro.tex
\section{Introduction}
\label{sec:intro}

For the fast and reliable evaluation of input-output relations
for parameterized partial differential equations
(\muPDE s),
reduced basis methods
have been developed over the last decade;
see, e.\,g., \cite{RoHuPa08,QaMaNe15}
or \cite[Chapter~19]{Qua14}
for comprehensive reviews, with the first reduced basis problem being investigated in the 1980's \cite{NoPe80}.
The methodology has been applied successfully to many different problem classes
both in the real-time and the many-query context.
These problem classes include among others finite element discretizations of elliptic equations \cite{RoHuPa08},
parabolic equations \cite{GrPa05,RoMaMa06, UrAn14} and
hyperbolic equations \cite{HaOh08, DaPlWe14}. Furthermore the reduced basis method has been extended to Stokes problems \cite{LoMaRo06,RoVe07,IQRV14, RoHuMa13} as well as to variational inequalities \cite{HaSaWo12, GU14} with a time-space formulation of the problem and corresponding analysis. 
It was also applied to Stochastic Processes in \cite{UrWi12, ElLi13} and to a finite volume scheme of a parameterized and highly nonlinear convection-diffusion problem with discontinuous solutions in 
\cite{DrHaOh11}.

A~posteriori error estimators w.\,r.\,t.\ parameter variations
generally facilitate
the construction of reduced basis spaces
by greedy algorithms
as well as the certification of the outputs of the reduced models.
Different greedy methods for reduced basis and error estimators have been introduced in \cite{MaPaTu02,MaPaTu002,UrVoZe14}, also a greedy method for eigenvalues is introduced in  \cite{ChEhLe13}, and the convergence of greedy methods has been analyzed in \cite{BICoDaDePeWo11,BuMaPaPrTu12,DePeWo13}. Another way to construct a reduced space is the proper orthogonal decomposition (POD) method, as discussed, e.\,g., in \cite{KaVo07,Qua14}.

The problem class of parameterized elliptic eigenvalue problems (\muEVP s) is highly important but up to now only marginally investigated
in the context of reduced basis methods.
The first approach \cite{MaMaOlPaRo00} from the year 2000,
which is based on \cite{MaPaPe99} among others,
is restricted to the special case of an estimator for the first eigenvalue.
In the following publications \cite{Pau07,Pau07a,Pau08}, the method from \cite{MaMaOlPaRo00} is developed further to include several eigenvalues.

However,
both the analysis and the algorithms do not cover the case of multiple eigenvalues.
Quite often, the ``vectorial approach'',
i.\,e., the treatment of the eigenvectors $\left(u_i(\mu)\right)_{1\leq i\leq\anzEW}$ as an $(\text{FE dimension}  \cdot \anzEW)$-dimensional object
and building the approximation space accordingly (cf.\ \cite[Section 2.3.5]{Pau07}),
results in poor accuracy.
This is due to the fact that the possible savings from reduced problems of smaller size
seem marginal if achievable at all. In addition high-frequency information
can and should be exploited
for the approximation of low-frequency information,
an effect that is expected to become more and more important
with increasing number of desired eigenvalues.
In \cite{ZaVe13},
an elastic buckling problem is studied.
While the model reduction is carried out solely/primarily for a linear problem,
the eigenvalue problem appears only in a second step.

Here,
a new RB space is built from the eigenfunctions associated with the smallest eigenvalues
at the previously identified parameters.
A non-rigorous a~posteriori bound is then computed by comparison with a reduced space approximation of double size;
cf.\ also \cite{Veroy03}.
Furthermore a component based RB method is studied for eigenvalue problems in \cite{VaHuKnNgPa14}. 

Very recently an RB method for the approximation of single eigenvalues in the context of parameterized elliptic eigenvalue problems has been investigated in \cite{FuMaPaVe15}. The authors derive a bound for the error in the first eigenvalue which is assumed to be single.

The aim of this paper is
to develop a model reduction framework for elliptic \muEVP s.
The application scenario we target
is the vibro-acoustics of cross-laminated timber structures.
Here,
a parameter-dependent eigenvalue problem in linear elasticity,
where the
input parameters are the material properties of different structural components,
is to be solved many times during a design/optimization phase.
Since the main part of a vibro-acoustical analysis is the modal analysis, which not only takes the first eigenvalue into account, but all eigenvalues under a certain frequency depending on the problem under investigation, 
the outputs of interest are the $\anzEW$ smallest eigenvalues with corresponding eigenfunctions.
A characteristic feature of the considered \muEVP s
is the appearance of multiple eigenvalues.
In particular, the multiplicities depend on the parameters.

A particular challenge of the considered \muEVP s is the rather large number of outputs of interest $\anzEW$, which in our exemplary case ranges from two to twenty. 
We are interested in approximating these smallest eigenvalues ``simultaneously'' in the sense that
a single reduced space is constructed for the variational approximation of the eigenvalue problem
and that
the individual a~posteriori error estimators for the eigenvalues
use the online components provided offline. This allows us to generate an efficient and accurate simultaneous reduced basis approximation.
The large number of outputs of interest $\anzEW$ justifies an increased computational effort
by an increased dimension $N$
as compared to the standard single output case.
In particular,
any decent, i.\,e., sufficiently accurate,
approximation needs $N\gg\anzEW$.
This approach constitutes a significant difference to the one taken in \cite{FuMaPaVe15} as our goal is a reduced basis approximation not only of one eigenvalue, but of a series of eigenvalues, including eigenvalues with multiplicity greater than one. The parameter-dependence of the multiplicity of the eigenvalues constitutes a major challenge in this context and is included into both our analysis and our algorithms.

Furthermore our experiments show that, in a greedy algorithm,
it is usually not optimal to include the first $\anzEW$ eigenfunctions for a particular parameter, neither is it advisable to choose the same number of eigenfunctions for different eigenvalues.
This may be attributed to the fact that the smoothness of the input-output relation can vary strongly
with the different outputs of interest, i.\,e., the eigenvalues.
We rather suggest to choose maximizing parameters for $\anzEW$ different error estimators, as described in Sect.~\ref{ssec:greedy}.
The reduced approximation should be of comparable quality
for a broad range of frequencies,
although in structural acoustics
the accuracy requirements might decrease with increasing frequency.
Note that,
for the application scenario at hand,
the number of desired eigenpairs
is typically in the order of ten
for simple components
and even larger for geometrically more complex structures.

The main contributions of this paper are
the analysis of an asymptotically reliable error estimator
including the case of multiple eigenvalues
and a series of
algorithmic advancements.
Our numerical results demonstrate that
tailored greedy strategies
yield very efficient reduced basis spaces
for the simultaneous approximation of many eigenvalues for the considered problem class.

The rest of the paper is structured as follows:
In Sect.~2,
we describe the problem setting and 
introduce the reduced basis method for \muEVP s.
Sect.~3 is devoted to the a~posteriori error analysis w.\,r.\,t.\ parameter variations.
We also discuss how to evaluate the derived error estimators computationally.
In Sect.~4,
several greedy algorithms are presented.
We demonstrate the effectivity of our algorithms
by numerical examples with the application to linear elasticity
in Sect.~5.

%% file: problem.tex
\section{Problem setting}
\label{sec:problem}

\subsection{Parameterized eigenvalue problems in computational mechanics}

Let the computational domain $\Omega \subset \RR^{d}$, with $d=2,3$, be bounded and polygonal.
As an elliptic eigenvalue model problem, we consider the linear elasticity case. But all our results also hold true for more general elliptic systems.
Then,
the eigenvalue problem in linear elasticity is given by
\begin{eqnarray}\label{eq:EVP1}
 -\text{div }\sigma = \lambda \rho u\quad\text{in }\Omega
\end{eqnarray}
with boundary conditions  prescribed as Dirichlet conditions on a closed 
non-trivial subset $\Gamma_D$ of $\partial\Omega$
and homogeneous Neumann conditions on $\partial\Omega\setminus\Gamma_D$. In addition,
the linearized stress and strain tensors are defined as
\begin{equation*}
\sigma(u)=\mathbb{C}(\mu) \epsilon(u)\qquad \text{ and } \qquad \epsilon(u) = \frac{1}{2} (\nabla u + \nabla u^ T),
\end{equation*}
respectively.
We set the density $\rho$ to $1$ for simplicity.
Furthermore,
the set of admissible parameters
is denoted by $\mathcal P\subset\RR^P$
and $\mu\in\mathcal P$ stands for a vector of parameters.
Then,
$\mathbb{C}(\mu)$ denotes the parameter-dependent Hooke's tensor,
which we assume to be uniformly positive definite.
To this end,
let $\Omega$
be decomposed into non-overlapping subdomains such that $\overline{\Omega}= \bigcup_{s}\overline\Omega_{s}$.
We assume that the 
material parameters
are piecewise constant w.\,r.\,t.\ this decomposition.
In the isotropic case,
the parameters may be chosen as
Young's modulus $E$ and Poisson's ratio $\nu$
such that
$P$ equals two times the number of structural components (i.\,e., subdomains).
More precisely,
we set $\mu_{2{s}-1} = E|_{\Omega_{s}}$ and $\mu_{2{s}} = \nu|_{\Omega_{s}}$ in this case.
The anisotropic case is treated analogously.

Let the bilinear forms $a(\cdot,\cdot;\mu): (H^1(\Omega))^d\times (H^1(\Omega))^d\to\RR$ and $m(\cdot,\cdot): (L^2(\Omega))^d\times (L^2(\Omega))^d\to\RR$ be given by
\begin{equation*}
 (u,v)\mapsto a(u,v;\mu):= \int_{\Omega} \mathbb{C}(\mu)\epsilon (u):\epsilon (v) \,\mathrm{d}x
\end{equation*}
and
\begin{equation*}
(u,v)\mapsto m(u,v) := (u,v)_{L^2(\Omega)} :=\int_{\Omega} u\cdot v\,\mathrm{d}x.
\end{equation*}
Note that
$a(\cdot,\cdot;\mu)$ depends on the parameter vector $\mu$
whereas $m(\cdot,\cdot)$ and $\Omega$ do not.

\begin{rmrk}
The equations of linear elasticity are used as a model problem as we are interested in the applications of vibro-acoustics. However this does not pose any restriction to the theoretical results shown in the following. Thus we could replace $a(\cdot,\cdot,\mu)$ by any $H^1$ elliptic bilinear form.
\end{rmrk}

Let $V\subset\left\{v\in (H^1(\Omega))^d\;|\;v|_{\Gamma_D}=0\right\}$ be a fixed conforming finite element space of dimension $\mathcal N$.
Then,
the discrete variational formulation of (\ref{eq:EVP1})
reads as:
Find the eigenvalues $\lambda(\mu) \in \RR$ and the eigenfunctions $u(\mu) \in V$ such that
\begin{equation}\label{eq:EVP}
  a(u(\mu),v;\mu) = \lambda(\mu) m(u(\mu),v) \quad \forall\;v\in V
\end{equation}
for given $\mu\in\mathcal P$.
We assume that the eigenvalues are positive and numbered as
\begin{equation*}
   0 < \lambda_1(\mu)\leq\ldots\leq\lambda_{\mathcal N}(\mu).
\end{equation*}
The corresponding eigenfunctions are denoted by
$u_i(\mu)\in V$ for $i=1,\ldots,\mathcal N$
with the normalization
\begin{equation*}
  m(u_i(\mu),u_j(\mu)) = \delta_{ij}\text{ for }1\leq i,j\leq\mathcal N.
\end{equation*}
In the present context,
the error of the finite element solution is assumed to be very small.
This is achieved by a fine mesh size leading to a large dimension $\mathcal N$.
The discretization error analysis can be found,
e.\,g., in \cite{BaOs91,BaGuOs89,BaOs89}.

Let $L\geq 1$ be the number of distinct eigenvalues of (\ref{eq:EVP}).
For multiple eigenvalues,
we use the standard notation from \cite{BaOs91}
and denote 
the lowest index of the $i$-th distinct eigenvalue
by $k_i$ and its multiplicity by $q_i$,
$i=1,\ldots,L$.
We write
$\Kcal_i:= \{k_i,\ldots,k_i+q_i-1\}$.
(Here and in the following, the dependency of the index notations on $\mu$ is suppressed as it is always clear from the context.)
The corresponding eigenspaces are denoted by
\begin{equation*}
  U_i(\mu) := \mathrm{span}\left\{u_{k_i}(\mu),\ldots,u_{k_i+q_i-1}(\mu)\right\}.
\end{equation*}

Now, the goal is to find a computationally inexpensive but accurate surrogate model that can be used in the many-query or real-time context.

\subsection{Model reduction}

We consider a variational approximation of the \muEVP\ in an $N$-dimensional reduced space
\begin{equation}\label{eq:V_N}
  V_N := \text{span}\left\{ \zeta_n\;|\; n=1,\ldots,N \right\} \subset V,
\end{equation}
$N\ll\mathcal N$.
As a matter of fact,
the choice of $V_N$ highly depends on the algorithmic methodology.
Several (snapshot-based) possibilities are investigated
in Sect.~\ref{sec:algo}.

Now, the ``reduced eigenvalue problem'' reads as
\begin{equation}\label{eq:redEVP}
  (\uRBnoindex(\mu),\lambdaRBnoindex(\mu))\in V_N\times\RR,\quad a(\uRBnoindex(\mu),v;\mu) = \lambdaRBnoindex(\mu) m(\uRBnoindex(\mu),v) \quad \forall\;v\in V_N
\end{equation}
for given $\mu\in\mathcal P$.
Let us emphasize that all eigenpairs of interest are approximated in the same space $V_N$.
As before,
we assume a numbering
$\lambdaRB{i}(\mu)$, $i=1,\ldots,N$
of the ``reduced eigenvalues''.
The minimum-maximum principles guarantee that
$\lambda_i(\mu) \leq \lambdaRB{i}(\mu)$ for $i=1,\ldots,N$;
see \cite[Sect.~8]{BaOs91}.
Note that the multiplicity of the finite element eigenvalues is not necessarily reflected in the reduced basis eigenvalues.
The corresponding eigenfunctions are denoted by
$\uRB{i}(\mu)\in V_N$ for $i=1,\ldots,N$,
again with the normalization
\begin{equation*}
  m(\uRB{i}(\mu),\uRB{j}(\mu)) = \delta_{ij}\text{ for }1\leq i,j\leq N.
\end{equation*}
In practice,
as mentioned before,
one is only interested in the first $\anzEW$ eigenvalues
for any chosen parameter.
We expect that the dimension $N$ required to achieve a certain accuracy will depend
not only on the smoothness of the parameter-dependency of the \muPDE\ but also on the number of outputs $\anzEW$.

In the present context,
$a(\cdot,\cdot;\mu)$ is affine w.\,r.\,t.\ the parameter $\mu$,
i.\,e.,
\begin{equation}\label{eq:affine_decomp}
  a(u,v;\mu) = \sum_{q=1}^Q \Theta_q(\mu) a_q(u,v)
\end{equation}
for suitable 
parameter-independent bilinear forms $a_q: (H^1(\Omega))^{\color{blue} d}\times (H^1(\Omega))^{\color{blue} d}\to\RR$
and coefficients $\Theta_q:\mathcal P\to\RR$,
which are readily derived from the constitutive equations.
For instance, we have two terms per subdomain in the isotropic case.
This leads to a fast online evaluation
as the cost of the assembly of the parameter-dependent reduced systems (i.\,e., matrices in $\RR^{N\times N}$ associated with (\ref{eq:redEVP}))
is independent of $\mathcal N$.
Note that the expansion (\ref{eq:affine_decomp}) will also be exploited for
an online-offline decomposition of the error estimators.

\subsection{Model reduction by proper orthogonal decomposition}\label{ssec:POD-RB}

Before turning to the development of greedy methods and a~posteriori error estimators for \muPDE s,
we illustrate the potential of model reduction techniques in the context of parameter dependent eigenvalue problems with multiple output values.

A common technique for model reduction is the proper orthogonal decomposition (POD) \cite{KaVo07,Qua14,QaMaNe15},
which yields the best possible reduced space
(in the sense that,
for a given series of snapshots,
the projection error w.\,r.\,t.\ the $L^2$-norm is minimized).
To this end,
let $S\subset V$ be a set of snapshots
generated by solving the \muEVP\
(each time for $\anzEW$ eigenfunctions)
for all parameters in a sufficiently large training set 
$\Xi_\mathrm{train}^\mathrm{POD}$.
Then,
in the definition (\ref{eq:V_N}) of the reduced space $V_N$,
orthonormalized functions $\left\{\zeta_1,\ldots,\zeta_N\right\}\subset\text{span}\{S\}$
are selected such that
\begin{equation*}
  \sum_{v\in S}\| v-\Pi_N v \|_{L^2(\Omega)}^2
\end{equation*}
is minimal.
Here,
$\Pi_N$ is the $L^2$-orthogonal projection to $\text{span}\left\{\zeta_1,\ldots,\zeta_N\right\}$.
This essentially amounts to
assembling the $\#S\times\#S$-correlation matrix of the snapshots in $S$ w.\,r.\,t.\ $(\cdot,\cdot)_{L^2(\Omega)}$
and finding its $N$ largest eigenvalues and corresponding eigenvectors.
For a detailed description of the usage of POD methods in the present context,
see, e.\,g.,
\cite{Qua14,KaVo07,QaMaNe15}.

\begin{figure}[th]
\begin{center}
\includegraphics[width=.33\textwidth]{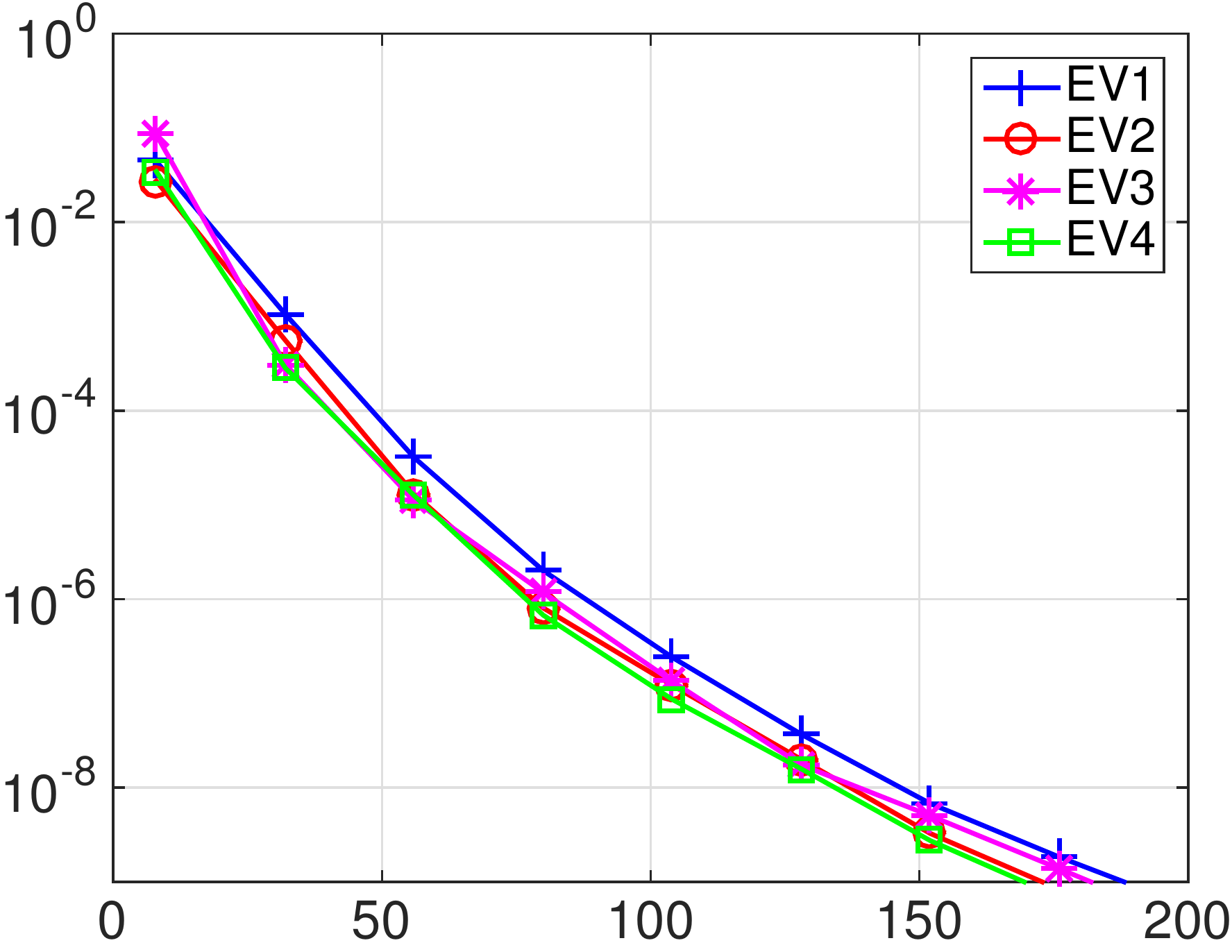}
\includegraphics[width=.33\textwidth]{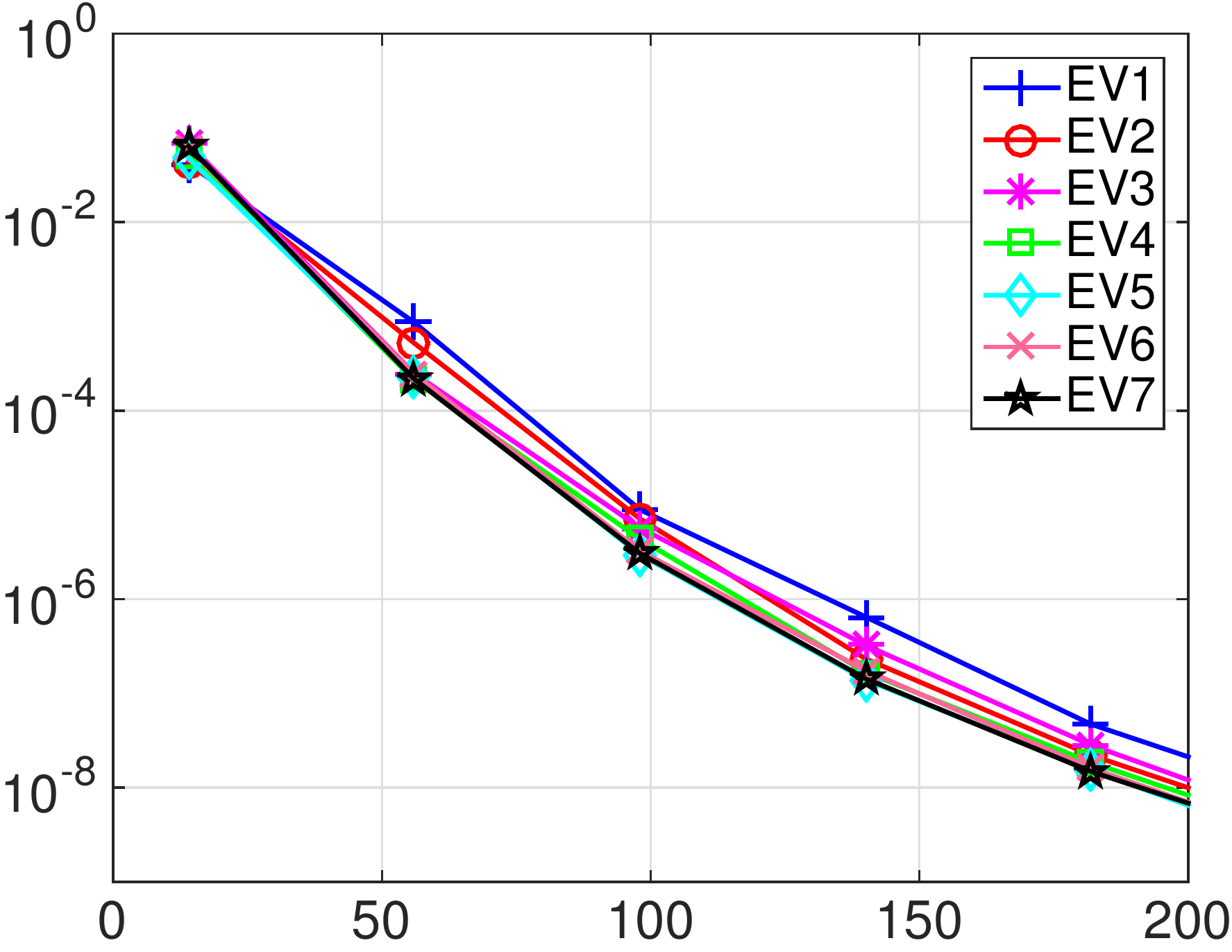}
\end{center}
\caption{Convergence of POD methods ($0<N\leq 200$) for the described \muEVP\ with different numbers of outputs of interest:
  Average relative errors in the eigenvalues $\lambda_1,\ldots,\lambda_K$
  for $\anzEW=4$ (left) and $\anzEW=7$ (right)}
\label{fig:PODRB_K47}
\end{figure}

Fig.~\ref{fig:PODRB_K47} illustrates the convergence of the POD method
for $\anzEW=4$ and $\anzEW=7$.
(The details of the underlying numerical experiment are elaborated in Sect.~\ref{sec:numerics}.)
On the one hand,
the results show that it is in principle possible to construct one single reduced space that effectively captures the parameter-dependent behavior of the first $\anzEW$ eigenfunctions simultaneously.
On the other hand,
it is evident that the RB dimension $N$ required for a certain accuracy increases with $\anzEW$.
More precisely,
the asymptotic decay of the error is approximately $C_4\,\text{e}^{-0.0513\cdot N}$ for $K=4$
and $C_7\,\text{e}^{-0.0477\cdot N}$ for $K=7$ for some constants $C_4$ and $C_7$.
From the cost point of view the POD is quite expensive, and thus we focus on computationally efficient greedy strategies in combination with a posteriori error bounds.

%% file: errorestimator.tex
\section{A~posteriori error estimation}
\label{sec:estimator}
In this section, we will establish  a~posteriori error estimators for our output quantities, in this case the eigenvalues. It is important to determine such estimators in order to find out which basis functions should be selected by the greedy method.
In particular, their computational evaluation must only depend on the basis size $N$ but not on the dimension of the finite element space $\mathcal N$. 

To this end,
we first
derive error bounds that still depend on the finite element eigenvalues, and in particular on their multiplicities.
Then, Sect.~\ref{ssec:approxerrorbound} is devoted to a computable approximation
yielding the desired error estimators.

In Sect.~\ref{ssec:estimatordecomp},
we recall a standard online-offline decomposition.

\subsection{Error bounds}\label{ssec:errorbound}

Let the parameter-dependent energy norm be defined as
$\aNorm{\cdot}:=a(\cdot,\cdot;\mu)^\frac12$.

In addition to a parameter dependent norm we are using a parameter indepenent norm defined as
$\aNormRef{\cdot}:=\aRef(\cdot,\cdot)^\frac12:=a(\cdot,\cdot;\muRef)^\frac12$.
For a linear functional $r:V\to\mathbb R$,
we define the corresponding dual norms by
\begin{equation*}
  \aNormDual{r} := \sup_{0\not=v\in V}\frac{r(v)}{\aNorm{v}}\quad\text{and}\quad  \aNormRefDual{r} := \sup_{0\not=v\in V}\frac{r(v)}{\aNormRef{v}},
\end{equation*}
respectively.

The analysis and the practical implementation employ different error representations,
namely the so-called reconstructed errors
w.\,r.\,t.\ the bilinear forms $a(\cdot,\cdot;\mu)$ and $\aRef(\cdot,\cdot)$.
Using the residual
\begin{equation*}
  v\mapsto r_i(v;\mu) := a(\uRB{i}(\mu),v;\mu)-\lambdaRB{i}(\mu)\, m(\uRB{i}(\mu),v)
\end{equation*}
for $i=1,\ldots,N$,
we define $e_i(\mu)\in V$ and $\eRef_i(\mu)\in V$ by
\begin{equation}\label{def:e}
  a(e_i(\mu),v;\mu) = r_i(v;\mu)\quad\forall\;v\in V
\end{equation}
and
\begin{equation}\label{def:eRef}
  \aRef(\eRef_i(\mu),v) = r_i(v;\mu)\quad\forall\;v\in V,
\end{equation}
respectively.
In particular,
$\aNormDual{r_i(\cdot;\mu)} = \aNorm{e_i(\mu)}$
and
$\aNormRefDual{r_i(\cdot;\mu)} = \aNormRef{\eRef_i(\mu)}$.

For any $\mu\in\mathcal P$, assume that $g(\mu)>0$ is a generalized coercivity constant
such that $g(\mu)\hat{a}(v,v)\le a(v,v;\mu)$ for all $v\in V$.
(Technically speaking,
$g(\mu)$ is the parameter-dependent coercivity constant of $a(\cdot,\cdot;\mu)$ w.\,r.\,t.\ $\aNormRef{\cdot}$.)
This implies
\begin{equation}\label{ineq:g1}
  \aNormDual{r} \leq g(\mu)^{-\frac12}\aNormRefDual{r}
\end{equation}
for any $r\in V'$
and
\begin{equation}\label{ineq:g2}
  \aNormRef{v} \leq g(\mu)^{-\frac12}\aNorm{v}
\end{equation}
for any $v\in V$.

We are now ready to prove the error bounds.
The following theorem,
combined with the computational/algorithmic aspects
in Sect.~\ref{ssec:approxerrorbound}
and Sect.~\ref{sec:algo},
generalize the results of \cite{Pau07,MaMaOlPaRo00,Pau08,Pau07a} for the case of multiple eigenvalues.

\begin{thrm}\label{thm:errorbound}
  Let $1\leq i\leq L$ such that
  $k_i+q_i-1\leq N$. 
  For $j=1,\ldots,q_i$,
  set
  \begin{equation}\label{def:dtilde}
    \dtilde_{k_i+j-1}(\mu):=\min_{\substack{ \mathcal{N}\ge l>k_i+q_i-1}}\left|\frac{\lambda_l(\mu)-\lambdaRB{k_i+j-1}(\mu)}{\lambda_l(\mu)}\right|.
  \end{equation}
  Then,
  \begin{equation}\label{eq:mainthm1}
     0 \leq \lambdaRB{k_i+j-1}(\mu) - \lambda_{k_i}(\mu) \le \frac{\aNormDual{r_{k_i+j-1}(\cdot;\mu)}^2}{\dtilde_{k_i+j-1}(\mu)}(1+  \frac{\aNormDual{r_{k_i+j-1}(\cdot;\mu)}}{\dtilde_{k_i+j-1}(\mu)^2 \sqrt{\lambda_{k_i}}}).
  \end{equation}
\end{thrm}

\begin{proof}
  Fix $\mu\in\mathcal P$, $1\leq i\leq L$ and $1\leq j\leq q_i$.
  Let $\uRB{k_i+j-1}(\mu) = \sum_{l=1}^\mathcal N \alpha_l u_l(\mu)$
  and $e_{k_i+j-1}(\mu) = \sum_{l=1}^\mathcal N \beta_l u_l(\mu)$.
  By (\ref{def:e}), we find
  \begin{equation*}
    \beta_l =\alpha_l \frac{\lambda_l(\mu)-\lambdaRB{k_i+j-1}(\mu)}{\lambda_l(\mu)}.
  \end{equation*}
  Therefore, we get
  \begin{eqnarray}
    \aNormDual{r_{k_i+j-1}(\cdot;\mu)}^2 &=& \sum_{l=1}^\mathcal N \alpha_l^2 \left(\frac{\lambda_l(\mu)-\lambdaRB{k_i+j-1}(\mu)}{\lambda_l(\mu)} \right)^2 \lambda_l(\mu) \nonumber\\
    &\geq& \sum_{\substack{ l>k_i+q_i-1}} \alpha_l^2 \left(\frac{\lambda_l(\mu)-\lambdaRB{k_i+j-1}(\mu)}{\lambda_l(\mu)} \right)^2 \lambda_l(\mu) \nonumber\\
    &\geq& \dtilde_{k_i+j-1}(\mu)^2\sum_{\substack{ l>k_i+q_i-1}} \alpha_l^2 \lambda_l(\mu).\label{est:anorm_r_1}
  \end{eqnarray}
  { 
  Using the fact that 
  $\sum_{l}\alpha_l^2=1$,  $\lambda_l(\mu) \le \lambda_{k_i}(\mu)$ for $l\le k_i+q_i-1$ , we find for the difference between approximated and detailed eigenvalue

  \begin{eqnarray}\label{eq:lambda_eq1}
    \Delta \lambda_{k_i}:=\lambdaRB{k_i+j-1}(\mu) - \lambda_{k_i}(\mu) &=& a(\uRB{k_i+j-1}(\mu),\uRB{k_i+j-1}(\mu);\mu) -  \lambda_{k_i}(\mu) \nonumber \\
    &=& \sum_{l=1}^\mathcal{N} \alpha_l^2 \lambda_l(\mu) - \lambda_{k_i}(\mu) \nonumber \\
    &=&  \sum_{l\le k_i+q_i-1} \alpha_l^2 (\lambda_l(\mu) - \lambda_{k_i}(\mu)) +  \sum_{l>k_i+q_i-1} \alpha_l^2 (\lambda_l(\mu) - \lambda_{k_i}(\mu)) \nonumber \\
    &\le & \sum_{l>k_i+q_i-1} \alpha_l^2 (\lambda_l(\mu) - \lambda_{k_i}(\mu)). 
     \end{eqnarray}
     From this we obtain two upper bounds for $\Delta \lambda_{k_i}$. The first one follows trivially from the fact that $\lambda_{k_i}(\mu)>0$ and (\ref{est:anorm_r_1})
     \begin{equation}
      \Delta \lambda_{k_i} \le  \sum_{l>k_i+q_i-1} \alpha_l^2 \lambda_l(\mu) \le \frac{\aNormDual{r_{k_i+j-1}(\cdot;\mu)}^2}{\dtilde_{k_i+j-1}(\mu)^2}.
     \end{equation}
The second bound is based on the Cauchy-Schwarz inequality and on Young's inequality. 
In terms of,
 \begin{eqnarray*}\label{eq:younge_part}
 (\lambda_l(\mu)-\lambda_{k_i}(\mu))^2
 &\le &(1+\epsilon)(\lambda_l(\mu)-\lambdaRB{k_i+j-1}(\mu))^2 + (1+\frac{1}{\epsilon})     (\Delta \lambda_{k_i})^2
\end{eqnarray*}  
for $\epsilon>0$, we get from (\ref{est:anorm_r_1}) and (\ref{eq:lambda_eq1})

    \begin{eqnarray*}\label{eq:lambda_eq2}
    \Delta \lambda_{k_i} &=& \sum_{l>k_i+q_i-1} \alpha_l \frac{\lambda_l(\mu) - \lambda_{k_i}(\mu)}{\lambda_l(\mu)}\sqrt{\lambda_l(\mu)}\alpha_l \sqrt{\lambda_l(\mu)}  \nonumber\\
    &\le &\left(\sum_{l>k_i+q_i-1} \alpha_l^2 (\frac{\lambda_l(\mu) - \lambda_{k_i}(\mu)}{\lambda_l(\mu)})^2\lambda_l(\mu)\right)^{\frac{1}{2}} \left(\sum_{l>k_i+q_i-1} \alpha_l^2 \lambda_l(\mu)\right)^{\frac{1}{2}}\\
     &\le & 
    \frac{1}{\dtilde_{k_i+j-1}(\mu)}\aNormDual{r_{k_i+j-1}(\cdot;\mu)} \sqrt{(1+\epsilon)\aNormDual{r_{k_i+j-1}(\cdot;\mu)}^2  +(1+\frac{1}{\epsilon})\Delta \lambda_{k_i}^2  \sum_{l>k_i+q_i-1} \alpha_l^2 \frac{1}{\lambda_l(\mu)}}\\
&\le &\frac{1}{\dtilde_{k_i+j-1}(\mu)}\aNormDual{r_{k_i+j-1}(\cdot;\mu)} \sqrt{(1+\epsilon)\aNormDual{r_{k_i+j-1}(\cdot;\mu)}^2  +(1+\frac{1}{\epsilon})\frac{\Delta \lambda_{k_i}^2  }{\lambda_{k_i}(\mu)}}\\
&\le& \frac{1}{\dtilde_{k_i+j-1}(\mu)}\aNormDual{r_{k_i+j-1}(\cdot;\mu)}^2 \sqrt{1+\epsilon + (1+\frac{1}{\epsilon})\frac{\aNormDual{r_{k_i+j-1}(\cdot;\mu)}^2}{\dtilde_{k_i+j-1}(\mu)^4 \lambda_{k_i}(\mu)}}.
  \end{eqnarray*}
}
Setting $\epsilon=\frac{\aNormDual{r_{k_i+j-1}(\cdot;\mu)}}{\dtilde_{k_i+j-1}(\mu)^2  \sqrt{\lambda_{k_i}(\mu)}}$ gives the upper bound in (\ref{eq:mainthm1}). The lower bound follows directly from 
\cite[Sect.~8]{BaOs91}.  
  
\end{proof}

{
Besides the generalization to multiple eigenvalues,
let us point out that our bounds are sharper than the ones, e.\,g., in \cite[Prop.~1]{Pau08},
as the lowest order term in (\ref{eq:mainthm1}) is of the form
$\frac{\aNormDual{r_{k_i+j-1}(\cdot;\mu)}^2}{\dtilde_{k_i+j-1}(\mu)}$
rather than
$\frac{\aNormDual{r_{k_i+j-1}(\cdot;\mu)}^2}{\dtilde_{k_i+j-1}(\mu)^2}$.
}
Note that the error bounds in Thm.~\ref{thm:errorbound} still depend on the finite element solution
via the eigenvalues $\lambda_l(\mu)$ in (\ref{def:dtilde}).
{

\begin{rmrk}
It is also possible to give an upper bound for the eigenvectors by replacing $\dtilde_{k_i+j-1}(\mu)$ by $\hat{d}_{k_i+j-1}(\mu)$ defined as
\begin{equation*}
 \hat{d}_{k_i+j-1}(\mu):=\min_{ \mathcal{N}\ge l>k_i+q_i-1\;\vee\;l<k_i}\left|\frac{\lambda_l(\mu)-\lambdaRB{k_i+j-1}(\mu)}{\lambda_l(\mu)}\right|.
\end{equation*}
Using now $\Pi_i:V\to U_i(\mu)$ as the orthogonal projection w.\,r.\,t.\ the $L^2$-inner product, we define
$\bar v := \Pi_i(\uRB{k_i+j-1}(\mu)) = \sum_{ \mathcal{N}\ge l>k_i+q_i-1\;\vee\; l<k_i} \alpha_l u_l(\mu)$ and give the upper bound as
 \begin{eqnarray*}
   \aNorm{\uRB{k_i+j-1}(\mu) - \bar v}^2 &=& \aNorm{\sum_{\mathcal{N}\ge l>k_i+q_i-1\;\vee\;l<k_i} \alpha_l u_l(\mu)}^2
   \\&=& \sum_{\mathcal{N}\ge l>k_i+q_i-1\;\vee\;l<k_i} \alpha_l^2 \lambda_l(\mu)\le  \frac{\aNormDual{r_{k_i+j-1}(\cdot;\mu)}^2}{\hat{d}_{k_i+j-1}(\mu)^2}.
 \end{eqnarray*}
\end{rmrk}
}

\subsection{{Error estimators}}\label{ssec:approxerrorbound}

\vskip1em

{
We now derive
approximate error bounds that are computable
in the sense that they do not depend on the finite element solution.
To achieve this,}
it remains to approximate $\dtilde_i$,
which may be interpreted as a measure for the relative distance between neighboring/adjacent eigenvalues,
particularly to decide which of the indices to exclude from the minimum.
We point out that the dimension of the (detailed) eigenspace is not accessible.
The application scenario we have in mind features multiple eigenvalues
with their multiplicities depending on the parameter.
It is therefore impossible to determine the structure of the spectrum
(i.\,e., the indices $k_i$ or the index sets $\Kcal_i$) a~priori.

Recall that the first $\anzEW$ eigenvalues are the output quantities of interest.
Assume that the reduced basis method converges in the following sense:
For $\mu\in\mathcal P$ and $1\leq i\leq\anzEW$,
\begin{equation*}
  \lambdaRB{i}(\mu)\to\lambda_i(\mu)\text{ for }N\to\mathcal N.
\end{equation*}
In particular,
$\lambdaRB{j}(\mu)\to\lambda_{k_i}(\mu)$
for $N\to\mathcal N$
for $j\in\Kcal_i$.

Given the eigenvalues
$\lambdaRB{i}(\mu)$, $i=1,\ldots,K$,
of (\ref{eq:redEVP}),
we replace $\lambda_l(\mu)$ in (\ref{def:dtilde}) by $\lambdaRB{l}(\mu)$
and approximate $\Kcal_i$ by
\begin{equation*}
  \KcalRB{i} := \left\{1\leq j\leq\anzEW+\mathfrak{r} ; \; \, \left|\frac{\lambdaRB{j}(\mu)-\lambdaRB{i}(\mu)}{\lambdaRB{j}(\mu)}\right| < \epslambda  \right\}
\end{equation*}
for a chosen ``tolerance'' $\epslambda>0$ and with $\mathfrak{r}$ as the difference between the index of the first eigenvalue after the multiplicity of the $\anzEW$-th eigenvalue  and the $\anzEW$-th eigenvalue itself.
In the case that we know a priori the maximal multiplicity of all relevant eigenvalues for all parameters, we set $\mathfrak{r}$ equal to this value. Otherwise we select it adaptively during the initialization phase of the greedy method. More precisely, we start with $\mathfrak{r} =1 $ and increase it by one as 
long as $K + \mathfrak{r} \in  \KcalRB{i} $.
Thus $\#\KcalRB{i}$ will be  our best guess
for the multiplicity of the eigenvalue to which $\lambdaRB{i}(\mu)$ converges.
Then for $1\le i \le K$,
\begin{equation}\label{eq:d_i}
  \d_i(\mu):=\min_{\substack{l\not\in\KcalRB{i} \\ \anzEW+\mathfrak{r} \ge l > i}}\left|\frac{\lambdaRB{l}(\mu)-\lambdaRB{i}(\mu)}{\lambdaRB{l}(\mu)}\right|
\end{equation}
is the relative distance of $\lambdaRB{i}(\mu)$ to the reduced eigenvalues that are further away than the chosen tolerance $\epslambda$. The adaptive selection of $\mathfrak{r}$ guarantees that even for $i=K$ and multiple eigenvalues
$d_i(\mu) $ is easily computable and does not severely underestimate $\tilde d_i(\mu)$.

Finally,  since we are looking for an asymptotic estimator  for the relative error in the eigenvalues which is cheaply computable in the online-phase, we
 neglect the higher order term in (\ref{eq:mainthm1}).
{In addition,} the parameter-dependent norm $\aNormDual{\cdot}$ {is replaced} by the parameter-independent norm $\aNormRefDual{\cdot}$ {by (\ref{ineq:g1}), which introduces an additional factor $g(\mu)^{-1}$}. 
{
To summarize we can state the following corollary:
\begin{crllr}
Let $i=1,\ldots,\anzEW$ and $\lambdaRB{i}(\mu)\to\lambda_i(\mu)\text{ for }N\to\mathcal N$. Furthermore let $\KcalRB{i}$ be defined as above and the distance between neighboring eigenvalues $d_i(\mu)$ be given as in (\ref{eq:d_i}). Then the error estimator given by
\begin{equation}\label{eq:eta_i}
  \mu\mapsto\eta_i(\mu) := \frac{\aNormRefDual{r_i(\cdot;\mu)}^2}{g(\mu)\cdot\d_i(\mu)\cdot\lambdaRB{i}(\mu)}.
\end{equation}
is asymptotically reliable in the sense that 
\begin{equation*}
0 \leq \frac{\lambdaRB{k_i+j-1}(\mu) - \lambda_{k_i}(\mu)}{\lambda_{k_i}} \le C \eta_i(\mu),
\end{equation*}
with $C$ tending to one as $N$ tends to $\mathcal N$.
\end{crllr}
}

Note that the approximation in (\ref{eq:d_i}) is, in general,  less accurate for $i=\anzEW$.
This is because the space $V_N$ is built to approximate well the $\anzEW$ outputs, but for the $\anzEW$-th estimator we need the $(\anzEW+\mathfrak{r})$-th outputs with $\mathfrak{r}\ge1$, which are approximated only roughly.
The tolerance
$\epslambda$ has to be selected such that it reflects the desired accuracy  of the  RB approximation.

\subsection{Online-offline decomposition}\label{ssec:estimatordecomp}

All error estimator contributions may be decomposed as already outlined in \cite{MaMaOlPaRo00}.
Let
$(\zeta_n)_{1\leq n\leq N}$ be the orthonormal basis (w.\,r.\,t.\ $m(\cdot,\cdot)$) of $V_N$.
For $0\leq q,p\leq Q$
let $\ARef^{q,p}\in\mathbb R^{N\times N}$
with $\ARef^{q,p}_{n,m} := \aRef(\xi_n^q,\xi_m^p)$ for $1\leq n,m\leq N$
where
\begin{equation}\label{eq:xi_q_n}
  \aRef(\xi_n^q,v) = a_q(\zeta_n,v),\quad\forall\;v\in V,\;1\leq n\leq N,\;1\leq q\leq Q,
\end{equation}
and
\begin{equation}\label{eq:xi_0_n}
  \aRef(\xi_n^0,v) = m(\zeta_n,v),\quad\forall\;v\in V,\;1\leq n\leq N.
\end{equation}
In the following,
we identify the function $\uRB{i}(\mu)\in V_N$
and its vector representation w.\,r.\,t.\ the basis $(\zeta_n)_{1\leq n\leq N}$
such that
$\left(\uRB{i}(\mu)\right)_n$ denotes the $n$-th coefficient.
Then,
given a reduced eigenpair $\left(\uRB{i}(\mu),\lambdaRB{i}(\mu)\right)$, we have the error representation
\begin{equation*}
  \eRef_i(\mu) = \sum_{n=1}^N \sum_{q=1}^Q \Theta_q(\mu) \left(\uRB{i}(\mu)\right)_n \xi_n^q
  - \lambdaRB{i}(\mu) \sum_{n=1}^N  \left(\uRB{i}(\mu)\right)_n \xi_n^0
\end{equation*}
by (\ref{def:eRef}).
Consequently,
the main contribution of $\eta_i(\mu)$ decomposes into
\begin{equation*}
  \begin{split}
    \aNormRefDual{r_i(\cdot;\mu)}^2 = & \sum_{n=1}^N \sum_{m=1}^N \sum_{q=1}^Q \sum_{p=1}^Q
    \left(\uRB{i}(\mu)\right)_n \left(\uRB{i}(\mu)\right)_m \Theta_q(\mu) \Theta_p(\mu) \,\ARef^{q,p}_{n,m}\\
    & + \lambdaRB{i}^2(\mu) \sum_{n=1}^N \sum_{m=1}^N
    \left(\uRB{i}(\mu)\right)_n \left(\uRB{i}(\mu)\right)_m \,\ARef^{0,0}_{n,m}\\
    & -2\,\lambdaRB{i}(\mu) \sum_{n=1}^N \sum_{m=1}^N \sum_{q=1}^Q
    \left(\uRB{i}(\mu)\right)_n \left(\uRB{i}(\mu)\right)_m \Theta_q(\mu) \,\ARef^{q,0}_{n,m}.
  \end{split}
\end{equation*}

We recall that only a single reduced space is built for the approximation of all eigenvectors simultaneously. Thus the above decomposition uses the same offline ingredients for all $1\leq i\leq K$.
In particular,
the number $\anzEW$ of desired eigenpairs does not directly influence the complexity
(only via the reduced space dimension $N$).

%% file: algorithms.tex
\section{Algorithms / Basis construction}
\label{sec:algo}

In this section,
we present different greedy strategies
that employ the error estimators of Sect.~\ref{sec:estimator}
to build the reduced space in (\ref{eq:V_N}).
The advantage as compared to the POD method motivated in Sect.~\ref{ssec:POD-RB} is
that only relatively few finite element solutions of the \muEVP\ need to be computed.

Since we use a single space for the approximation of multiple outputs,
we have several natural possibilities
which are investigated in Sect.~\ref{ssec:greedy}.
Sect.~\ref{ssec:greedy_ext} is devoted to an extension that takes into account multiple eigenvalues.
In Sect.~\ref{ssec:smallPOD},
a remedy for the potential unreliability of the error estimators for small $N$ is discussed.

\subsection{Greedy selection of snapshots for single eigenvalues}\label{ssec:greedy}

Recall that the $\anzEW$ smallest eigenvalues are the quantities of interest,
where $\anzEW$ is typically 
$2-20$
for our application scenario.
In principle,
given a reduced space,
one could try to identify a suitable $\mu\in\mathcal P$
and then include the first $\anzEW$ eigenfunctions
for this parameter value.
(In each greedy step, this would require the detailed FE solution of (\ref{eq:EVP}) for one parameter only.)
However,
numerical studies clearly show that this naive choice is far from optimal
as the generated reduced spaces tend to be much too large.
This is because
the errors in the individual eigenvalues and eigenfunctions are only very weakly correlated,
if at all.
There are at least the following two much more natural options.

\begin{algorithm}[th]
  \normalsize
  \caption{Multi-choice greedy}\label{alg:greedy_alle}
  \begin{algorithmic}[1]
    \For{$i=1,\ldots,\anzEW$}
    \State $\zeta_i \gets u_i(\muRef)$
    \EndFor
    \State $N \gets \anzEW$
    \While{$N < N_\mathrm{max}$}
    \For{$i=1,\ldots,\anzEW$}
    \State $\mu_{\mathrm{max},i} \gets \argmax_{\mu\in\Xi_\mathrm{train}} \eta_i(\mu)$
    \If{$ \eta_i(\mu_{\mathrm{max},i}) > \TOL$ }
    \State $N \gets N+1$
    \State $\zeta_N \gets u_i(\mu_{\mathrm{max},i})$\;\,{(orthonormalized)}
    \EndIf
    \EndFor
    \If{$\max_{\mu\in\Xi_\mathrm{train},1\leq i\leq\anzEW} \eta_i(\mu) < \TOL$}
    \State \textbf{break}
    \EndIf
    \EndWhile
  \end{algorithmic}
\end{algorithm}

\begin{algorithm}[th]
  \normalsize
  \caption{Single-choice greedy}\label{alg:greedy_einzeln}
  \begin{algorithmic}[1]
    \For{$i=1,\ldots,\anzEW$}
    \State $\zeta_i \gets u_i(\muRef)$
    \EndFor
    \State $N \gets \anzEW$
    \While{$N < N_\mathrm{max}$}
    \State $(\mu_\mathrm{max},i_\mathrm{max}) \gets \argmax_{\mu\in\Xi_\mathrm{train},1\leq i\leq\anzEW} \eta_i(\mu)$
    \State $N \gets N+1$
    \State $\zeta_N \gets u_{i_\mathrm{max}}(\mu_\mathrm{max})$\;\,{(orthonormalized)}
    \If{$\max_{\mu\in\Xi_\mathrm{train},1\leq i\leq\anzEW} \eta_i(\mu) < \TOL$}
    \State \textbf{break}
    \EndIf
    \EndWhile
  \end{algorithmic}
\end{algorithm}

Let a sufficiently rich training set $\Xi_\mathrm{train}\subset\mathcal P$ be given.
Then,
in Alg.~\ref{alg:greedy_alle},
the individual arg\,max for each $1\leq i\leq\anzEW$ is chosen separately.
In contrast,
Alg.~\ref{alg:greedy_einzeln} chooses 
only one single
arg\,max.
Note that both
Alg.~\ref{alg:greedy_alle} (line~7)
and
Alg.~\ref{alg:greedy_einzeln} (line~6)
require the evaluation of all error estimators at all parameters in $\Xi_\mathrm{train}$ to determine the choice of $\mu$. This does not lead to large computations since the calculations are only performed with the reduced space of size $N$, such that we obtain $\anzEW$ reduced eigenpairs for any $\mu\in\Xi_\mathrm{train}$,
see also Sect.~\ref{ssec:estimatordecomp}.
However,
Alg.~\ref{alg:greedy_alle} (line~10)
and
Alg.~\ref{alg:greedy_einzeln} (line~8)
require also finite element solutions which then determine the reduced basis space.

The multi-choice variant rests on the intuition
that the individual eigenfunctions can/should be approximated separately.
In contrast,
the single-choice variant takes into account
that the approximation power of eigenfunctions to large eigenvalues can be exploited also for eigenfunctions to smaller eigenvalues.

During the greedy procedure,
we orthonormalize the selected basis functions.
Not only does this yield small condition numbers of the reduced systems;
it is also beneficial for the special treatment of multiple eigenvalues
described in the next section.

In Alg.~\ref{alg:greedy_alle} (line~10)
and Alg.~\ref{alg:greedy_einzeln} (line~8),
an orthonormalization is performed.
For this purpose,
let $\Pi_N:V\to V_N$ be the $L^2$-orthogonal projection to the current reduced space.
For a snapshot candidate $\zeta\in V$
(i.\,e., one of the eigenfunctions chosen as described above),
we compute $\tilde\zeta := \zeta - \Pi_N\zeta$.
Then, if $\|\tilde\zeta\|$ is sufficiently large
($\geq\epsproj$),
the ``new contribution''
$\frac{\tilde\zeta}{\|\tilde\zeta\|}$
is included in the reduced basis;
see also Sect.~\ref{ssec:greedy_ext}.

\subsection{Extended selection for multiple eigenvalues}\label{ssec:greedy_ext}

In case of multiple eigenvalues,
the greedy method needs to be modified as follows.
Assume an index $1\leq \tilde\imath\leq\anzEW$
and a parameter $\tilde\mu$
have been selected
by means of the eigenvalue-based estimators $(\eta_i)_{i=1,\ldots,\anzEW}$,
 in Alg.~\ref{alg:greedy_einzeln} (line~6),
or several parameters $\tilde\mu$ have been selected in Alg.~\ref{alg:greedy_alle} (line~7),
such that the span of $u_{\tilde\imath}(\tilde\mu)$ is to be included in the reduced space.

However,
a large value of $\eta_{\tilde\imath}(\tilde\mu)$
merely indicates that the corresponding (fine) eigenspace $U_{\tilde\imath}(\tilde\mu)$ contains functions
that are badly approximated by the current reduced space.
Nevertheless the eigenspace might also contain other functions that are already well approximated.
Consequently,
if the detailed eigenvalue associated with a chosen snapshot has multiplicity greater than one,
we aim to add all the eigenfunctions for the multiple eigenvalue, except the ones which are already approximated well enough. A motivation for exploring the whole eigenspace for  multiple eigenvalues is to guarantee that we take the correct eigenvalue/eigenfunction, since we cannot ensure that the indexed eigenvalue/eigenfunction in the reduced space is the same as in the detailed calculation. This is due to the fact that there is no prescribed ordering for the eigenfunctions corresponding to a multiple eigenvalue.

As for the definition of $d_i(\mu)$
one has to compute 
a sufficient number $\anzEW' > \anzEW$ of eigenfunctions of the finite element
problem \muEVP\ (\ref{eq:EVP}) such that 
$\lambda_{\anzEW'}(\tilde\mu) / \lambda_\anzEW(\tilde\mu) > 1+\epslambda$.
Then,
lines~9--10 in Alg.~\ref{alg:greedy_alle} are replaced by:\\[-2mm]
\hrule
\begin{algorithmic}[0]
  \ForAll{$j\geq 1$ \textbf{with} $|\lambda_{j}(\mu_{\mathrm{max},i})-\lambda_{i}(\mu_{\mathrm{max},i})| / \lambda_{i}(\mu_{\mathrm{max},i}) \leq \epslambda$}
  \If{$\|u_{j}(\mu_{\mathrm{max},i})-\Pi_Nu_{j}(\mu_{\mathrm{max},i})\|_{L^2(\Omega)} \geq \epsproj$}
  \State $N \gets N+1$
  \State $\zeta_N \gets u_{j}(\mu_{\mathrm{max},i})$\;\,{(orthonormalized)}
  \EndIf
  \EndFor
\end{algorithmic}
\hrule
\vskip.5em
\noindent
Analogously,
lines~7--8 in Alg.~\ref{alg:greedy_einzeln} now read as:\\[-2mm]
\hrule
\begin{algorithmic}[0]
  \ForAll{$j\geq 1$ \textbf{with} $|\lambda_{j}(\mu_\mathrm{max})-\lambda_{i_\mathrm{max}}(\mu_\mathrm{max})| / \lambda_{i_\mathrm{max}}(\mu_\mathrm{max}) \leq \epslambda$}
  \If{$\|u_{j}(\mu_\mathrm{max})-\Pi_Nu_{j}(\mu_\mathrm{max})\|_{L^2(\Omega)} \geq \epsproj$}
  \State $N \gets N+1$
  \State $\zeta_N \gets u_{j}(\mu_\mathrm{max})$\;\,{(orthonormalized)}
  \EndIf
  \EndFor
\end{algorithmic}
\hrule
\vskip.5em
\noindent
Here,
$\Pi_N:V\to V_N$ denotes the $L^2$-orthogonal projection.
The parameter $\epsproj$ is a small tolerance that prevents the selection of functions that are already approximated sufficiently well.

\subsection{Initialization of the greedy method}\label{ssec:smallPOD}

{
In our calculations, we need an error estimator for the $\anzEW$-th eigenvalue. For the computation of this estimator, we need a rough approximation of the $(\anzEW + \mathfrak{r})$-th eigenvalue.
In order to ensure that our reduced space has the ability to roughly approximate this $(\anzEW + \mathfrak{r})$-th eigenvalue, we use an initial approximation space in which we include the corresponding components.
We suggest to include components using the proper orthogonal decomposition method
{described in Sect.~\ref{ssec:POD-RB} (with $N=N_\mathrm{init}$)}
applied to a small number of snapshots.
{Here, the snapshots $S\subset V$ are associated with a} training set $\Xi_\mathrm{train}^\mathrm{POD}$
typically of size $2^P$,
taking into account the extension described in Sect.~\ref{ssec:greedy_ext}.
This initial approximation space {of dimension $N_\mathrm{init}$}, which is constructed
as an initialization step for the greedy algorithm, should be sufficiently large as the reliability of the error estimators analyzed in Sect.~\ref{sec:estimator} can depend on the dimension of the reduced space. 
{
To make sure that we are able to calculate and to approximate the $(\anzEW + \mathfrak{r})$-th eigenvalues, we chose our 
$N_\mathrm{init}$ to be at least $(\anzEW + \mathfrak{r})$ times a factor $\ge 1.5$. 
}
}

%% file: numerics.tex
\section{Numerical results}
\label{sec:numerics}

In this section, the performance of the proposed algorithms is illustrated by numerical examples, { in two and in three dimensions. For the two-dimensional calculations we use plane strain elasticity while for the three-dimensional simulations we use linear elasticity.}
The implementation is performed in MATLAB based on the RBmatlab \cite{DrHaKaOh12} library.
We investigate the individual components and highlight their benefits in several steps.

\subsection{Preliminaries}

First, 
in Sect.~\ref{sec:extended_vs_non} to  Sect.~\ref{ssec:run_time},
we choose $\Omega$ as rectangle of size $3.0 \times 1.0$ with Dirichlet boundary on the left and on the right.
Let $\Omega$ be split into three subdomains of size $1.0 \times 1.0$. 
The material parameters $E$ and $\nu$ used for these subdomains are in the range of $10-100$ and $0.1-0.4$, respectively;
we have $P=6$ and $Q=6$.
We choose a uniform random sample of size $10,000$
as set of training parameters $\Xi_\mathrm{train}\subset\mathcal P$.
To evaluate the errors,
another sufficiently rich set of parameters $\Xi_\mathrm{test}\subset\mathcal P$ is used of size $1000$. For our initial space we choose $N_\mathrm{init} \le 40$, depending on the desired number of eigenvalues $\anzEW$.
We always report the average errors of the reduced approximations given by
\begin{equation*}
\frac{1}{\#\Xi_\mathrm{test}} \sum_{\mu \in \Xi_\mathrm{test}}\frac{\lambdaRB{i}(\mu) - \lambda_i(\mu)}{\lambda_i(\mu)}
\end{equation*}
and comment on the standard deviation at the end of Sect.~\ref{ssec:effectivity} in Remark~\ref{rem:statistics}.

For the generalized coercivity estimate,
we exploit the affine decomposition of the bilinear form
and set
\begin{equation}\label{eq:g_einfach}
  g(\mu) := \min_{q=1,\ldots,Q}\;\frac{\Theta_q(\mu)}{\Theta_q(\muRef)}.
\end{equation}
We emphasize that $g(\mu)$ merely relates the bilinear forms $a(\cdot,\cdot;\mu)$ and $\aRef(\cdot,\cdot)$;
a coercivity estimate for $a(\cdot,\cdot;\mu)$ itself is not required in the present context.
Note that (\ref{eq:g_einfach}) indeed yields an admissible parameter-dependent constant
provided the bilinear forms
$a_q(\cdot,\cdot)$
in (\ref{eq:affine_decomp})
are positive semi-definite
and the coefficient functions $\Theta_q(\cdot)$
in (\ref{eq:affine_decomp})
are positive;
see, e.\,g., \cite[Sect.~4.2.2]{PaRo06}.
This is true for our application.
Better results (i.\,e., a larger lower bound) could be obtained
by the more expensive successive constraint method \cite{HuRoSePa07}.
In the present setting,
the estimate (\ref{eq:g_einfach}) is typically smaller than the exact solution of the corresponding generalized eigenvalue problem by a factor ranging from $0.7$ to $0.98$.

\subsection{Extended selection vs.\ non-extended selection} \label{sec:extended_vs_non}

We first illustrate the necessity of the extended selection for multiple eigenvalues.
Fig.~\ref{fig:POD_adapt_nonadapt}
shows the behavior of a POD method
with (left) and without (right)
the extended selection described in Sect.~\ref{ssec:greedy_ext}
for the first two eigenvalues ($\anzEW = 2$).
Fig.~\ref{fig:Greedy_adapt_nonadapt} shows the same comparison for the greedy method
(Alg.~\ref{alg:greedy_einzeln}).
For both the POD method and the greedy method,
we observe that in the variants without extension
the convergence for the second eigenvalue becomes slower
after a certain number of basis functions has been included.
In contrast,
the extended selection yields convergence curves that approximately coincide.

\begin{figure}[th]
\centering
\includegraphics[width=.33\textwidth]{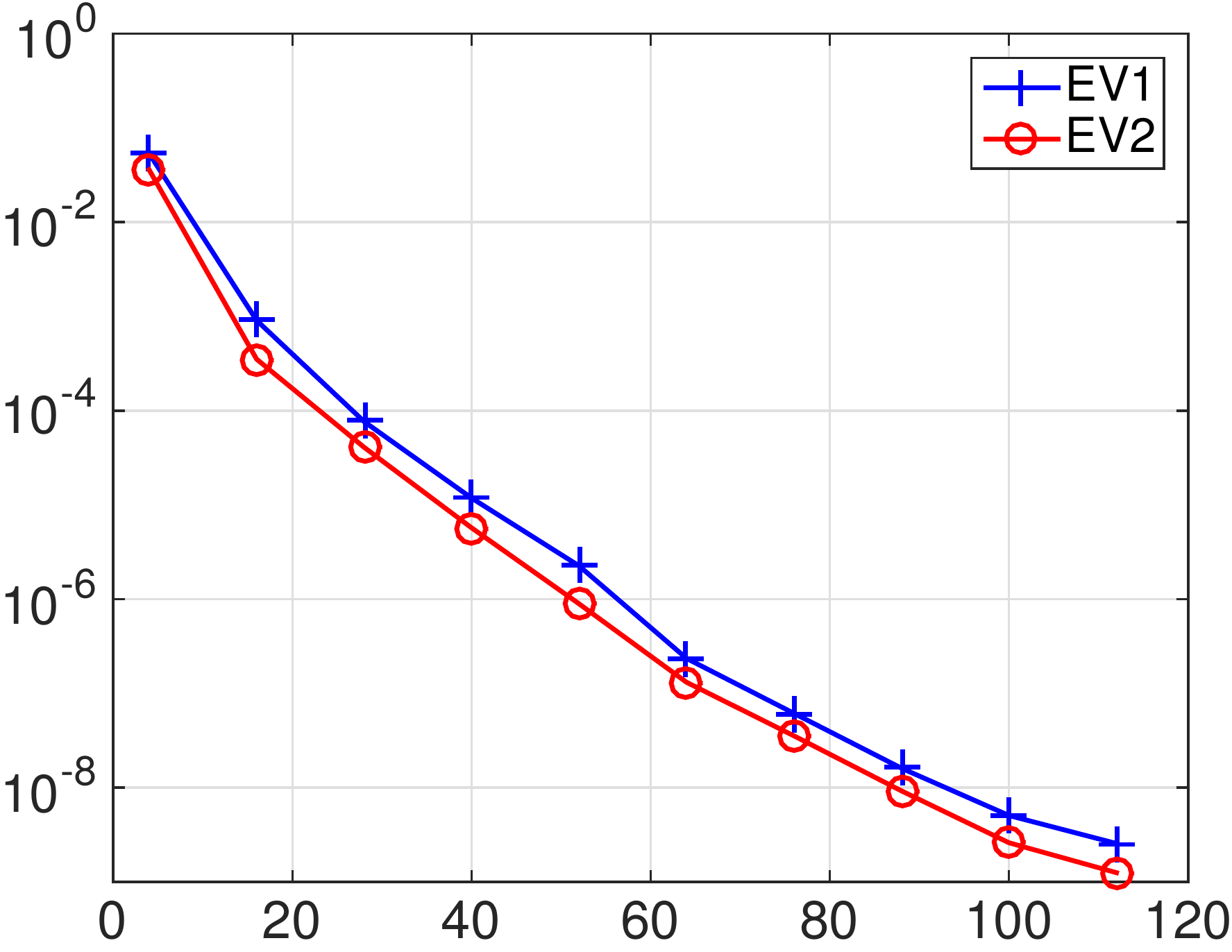} \hspace*{0.2cm}
\includegraphics[width=.33\textwidth]{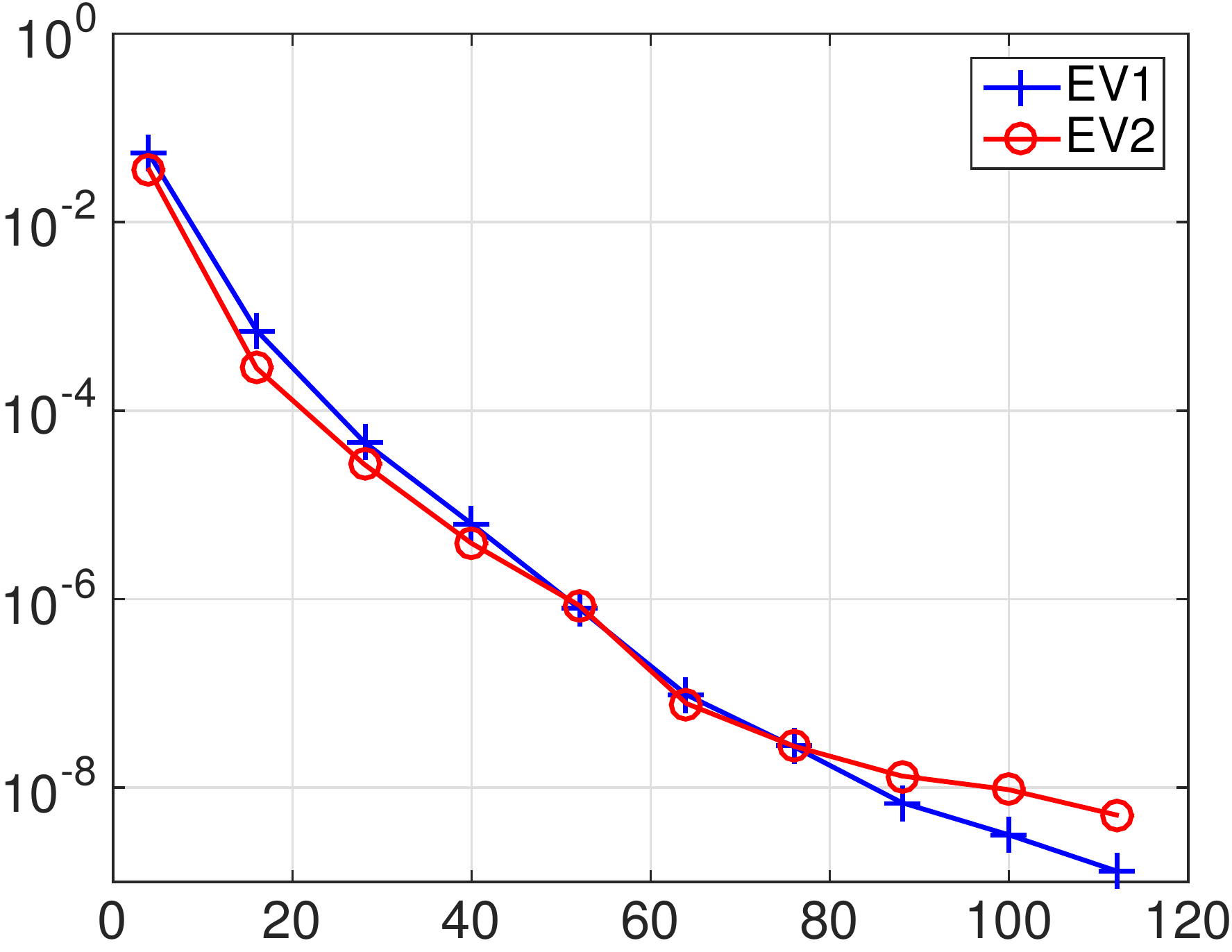} \\
\vspace*{0.2cm}
\includegraphics[width=.33\textwidth]{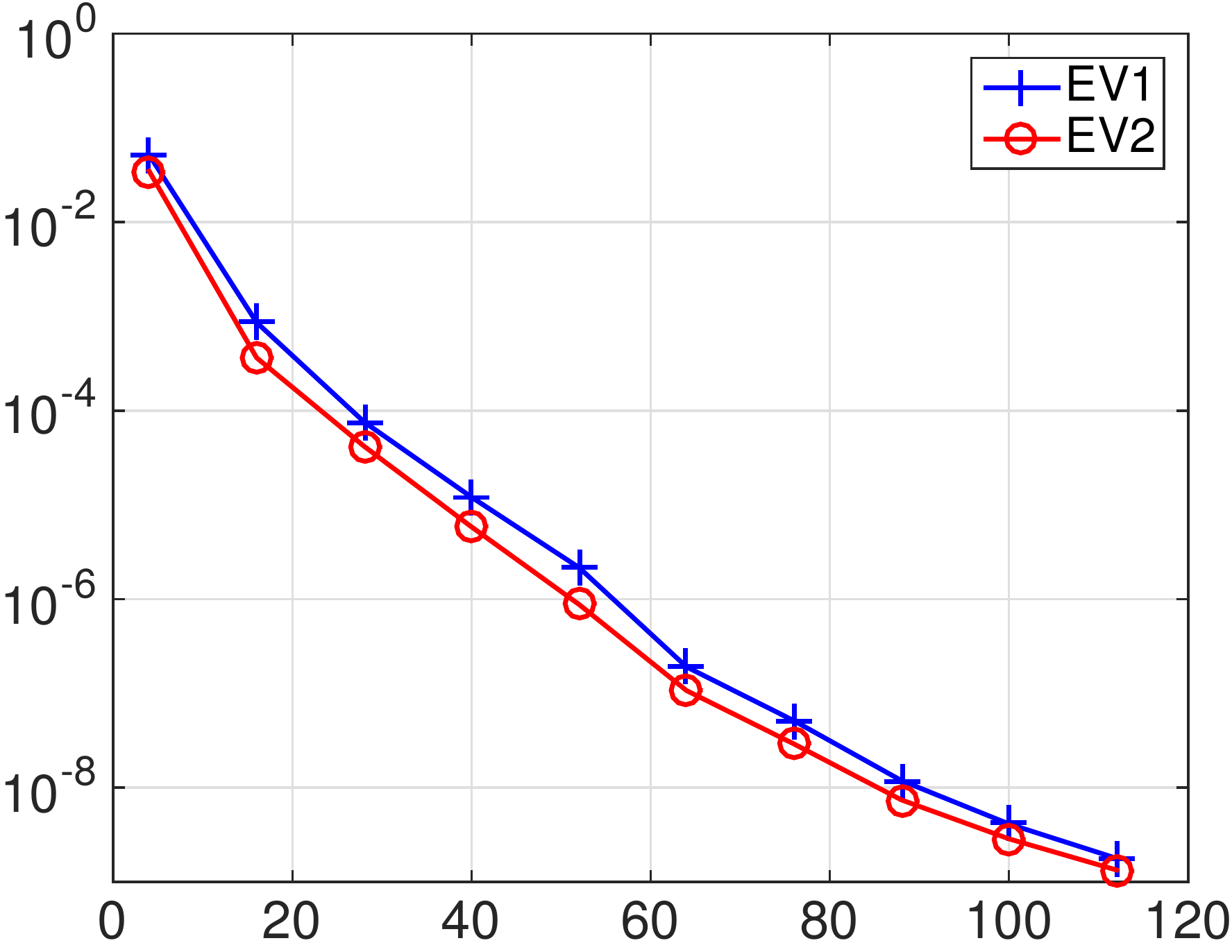} \hspace*{0.2cm}
\includegraphics[width=.33\textwidth]{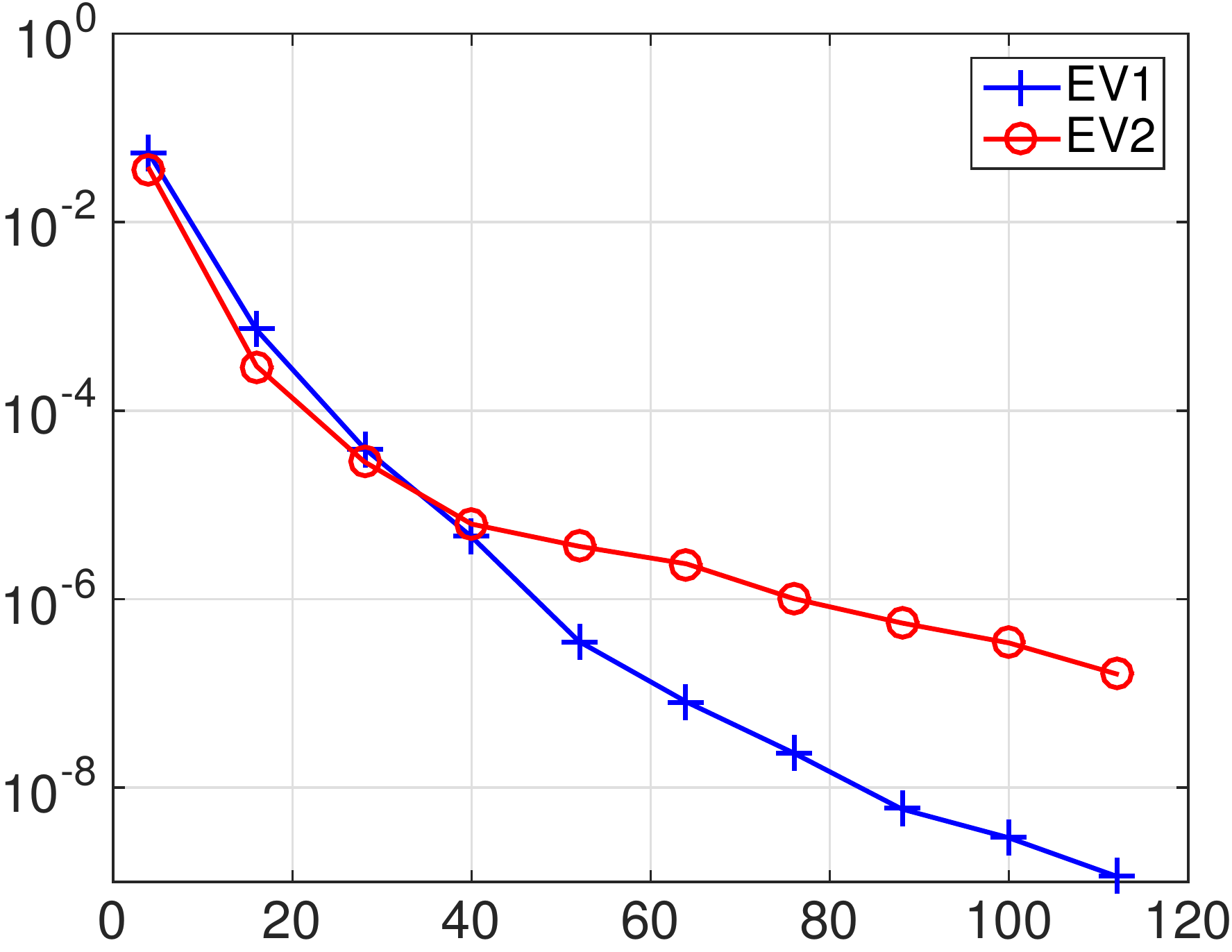} 
\caption{Error decay for the eigenvalues with the POD method: extended (left) vs.\ non-extended (right). First row: training size 10,000; second row: training size 1,000.}
\label{fig:POD_adapt_nonadapt}
\end{figure}

\begin{figure}[th]
\centering
\includegraphics[width=.33\textwidth]{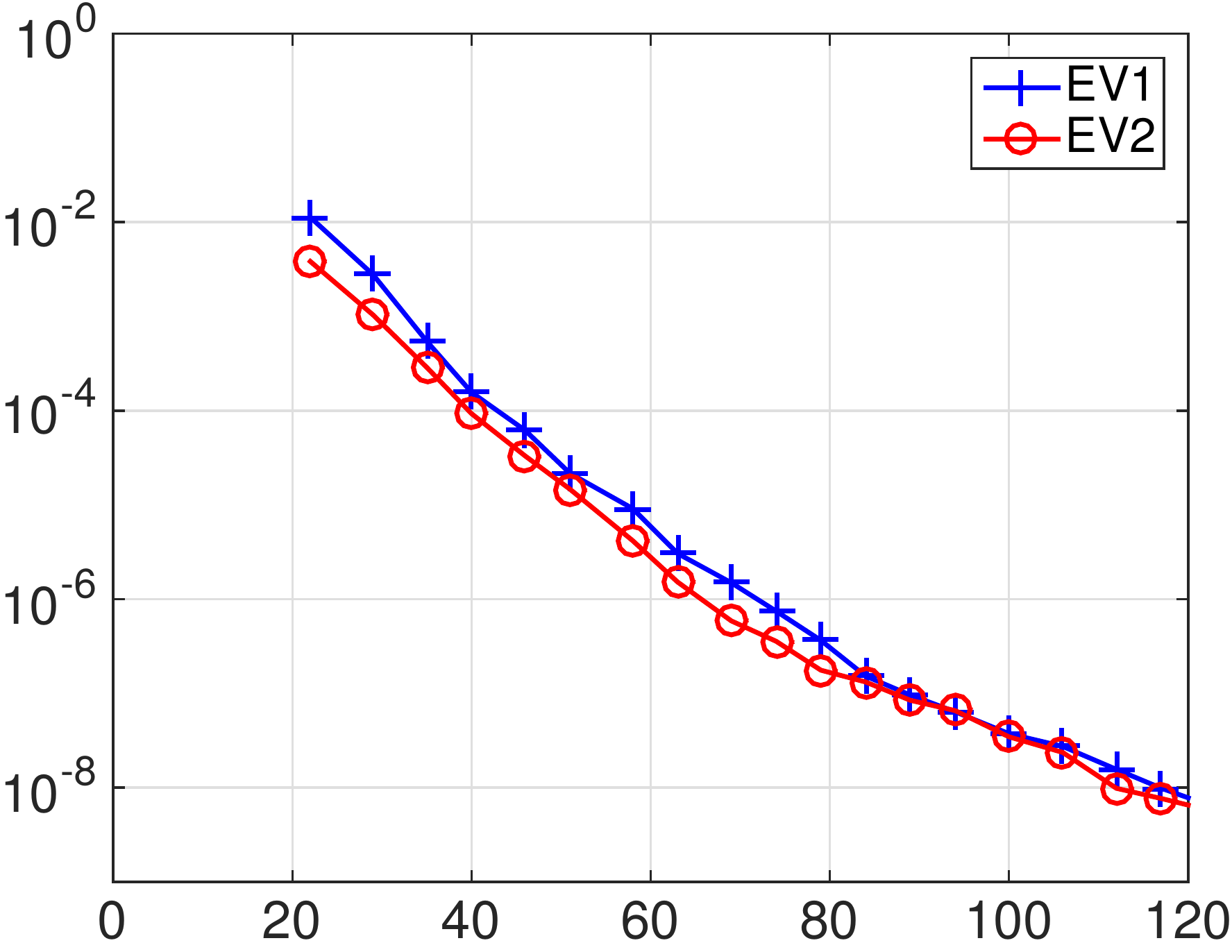} \hspace*{0.2cm}
\includegraphics[width=.33\textwidth]{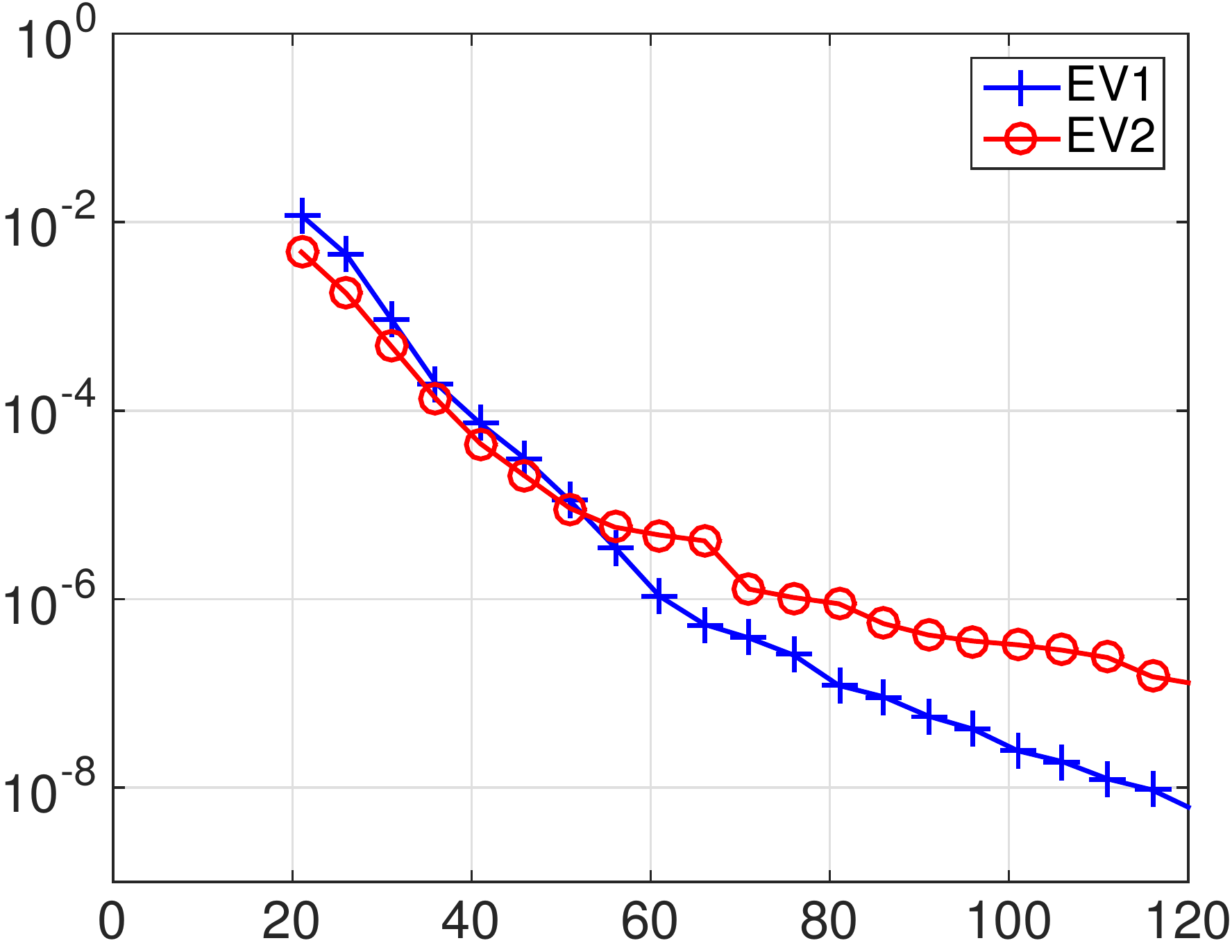} \\
\caption{Error decay for eigenvalues with the greedy method: extended (left) vs.\ non-extended (right); training size 1,000.}
\label{fig:Greedy_adapt_nonadapt}
\end{figure}

The shortcomings of the non-extended methods may be explained by the fact that the second eigenvalue
has multiplicity two for certain parameters and in these cases, for the multiple eigenvalue, the correct eigenfunction is not necessarily chosen.
Note that the effect is more significant for a smaller POD training size (Fig.~\ref{fig:POD_adapt_nonadapt}, second row) as it is less likely that all directions of an eigenspace are present in the snapshot set.
The convergence of the second reduced eigenvalue possibly improves drastically
if, incidentally, the missing component is added during the greedy method.

\subsection{Multi-choice vs.\ single-choice greedy method}

Here, we illustrate the benefit of Alg.~\ref{alg:greedy_einzeln} in comparison to Alg.~\ref{alg:greedy_alle}.
In Fig.~\ref{fig:alle-vs-einzeln},
for $\anzEW=4$,
one can see that with Alg.~\ref{alg:greedy_alle} (left)
the convergence behavior varies over the course of the greedy method
while with Alg.~\ref{alg:greedy_einzeln} (right) all desired eigenvalues exhibit similar convergence.
(This also holds true for the errors in the eigenfunctions not shown here.)

\begin{figure}[th]
\centering
\includegraphics[width=.33\textwidth]{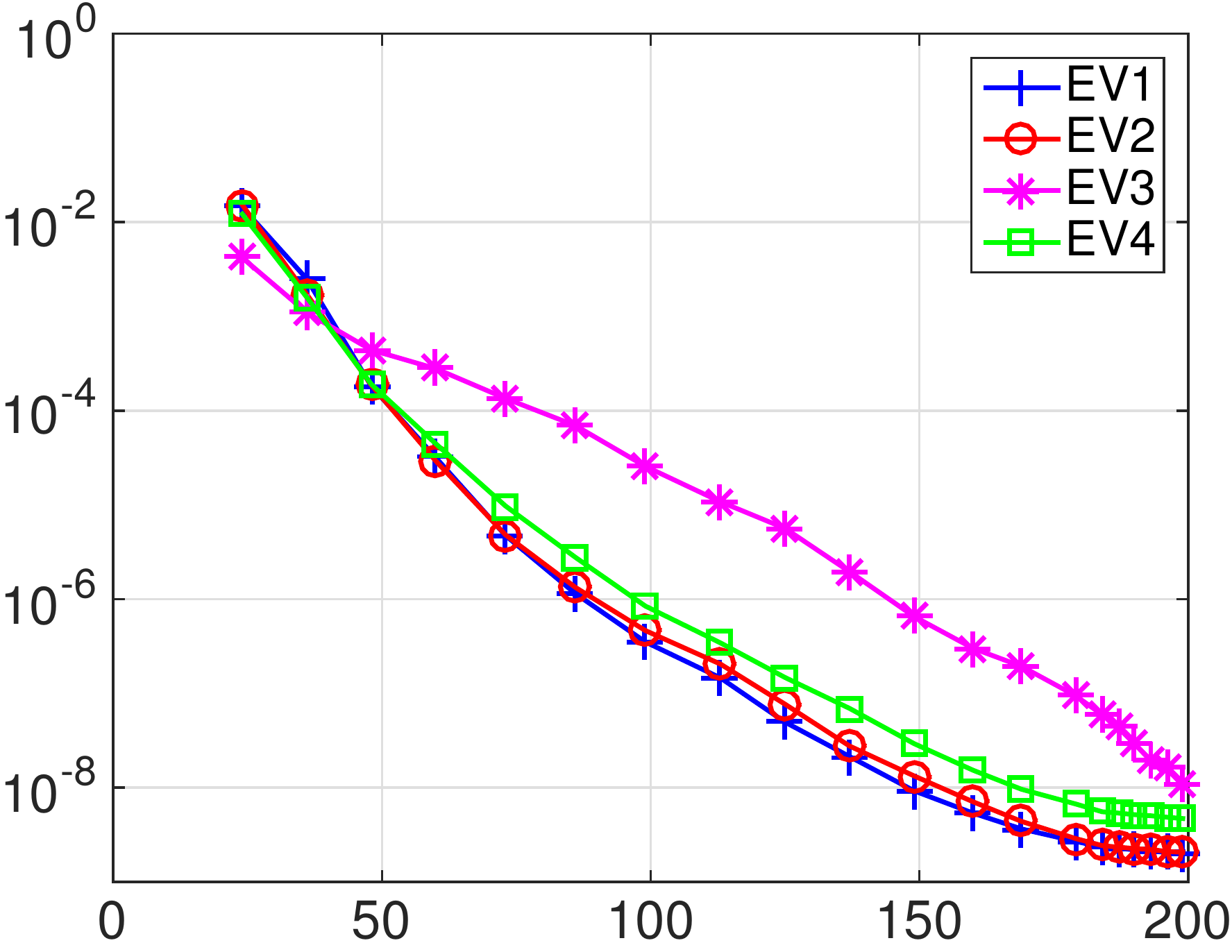}
\hspace*{2mm}
\includegraphics[width=.33\textwidth]{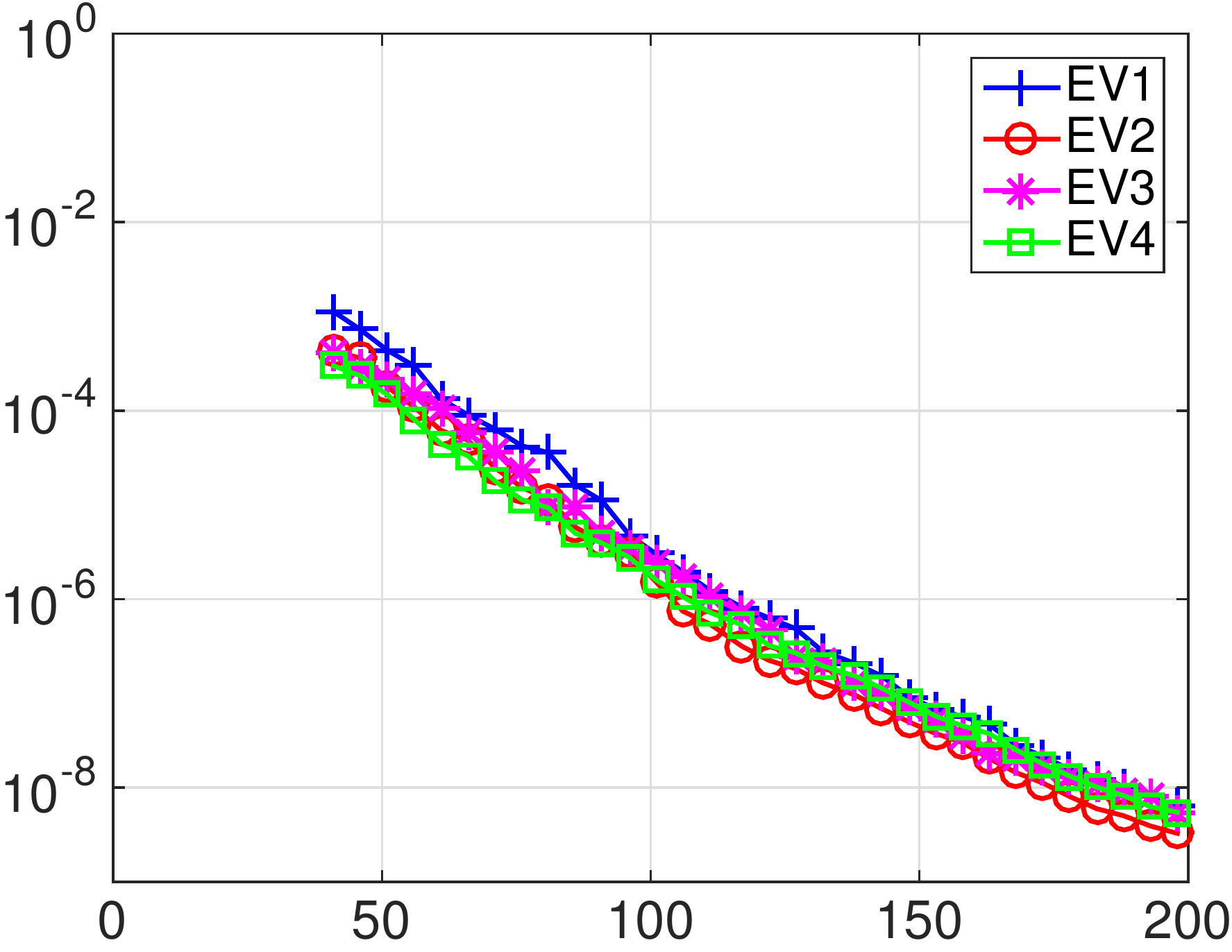}
\caption{RB error decay: comparison of Alg.~\ref{alg:greedy_alle} (left) and Alg.~\ref{alg:greedy_einzeln} (right)
  for $\anzEW=4$}
\label{fig:alle-vs-einzeln}
\end{figure}

The poor convergence of the third eigenvalue
only improves rapidly at $N\approx 170$,
after the other three eigenvalues have reached an accuracy in the order of the target tolerance,
and thus the algorithm only chooses EV\,3.
This effect
(namely an imbalanced resolution of the relevant eigenspaces during the greedy method)
is directly related to the 
inappropriate a~priori assumption of Alg.~\ref{alg:greedy_alle}
that roughly the same number of snapshots corresponding to the first $\anzEW$ eigenvalues should be included in the reduced space.
{
At this point it should also be noted that in general Alg.~\ref{alg:greedy_alle} creates a larger RB space than Alg.~\ref{alg:greedy_einzeln} as soon as more eigenvalues have a poor convergence.
}
\begin{figure}[th]
\centering
\includegraphics[width=.32\textwidth]{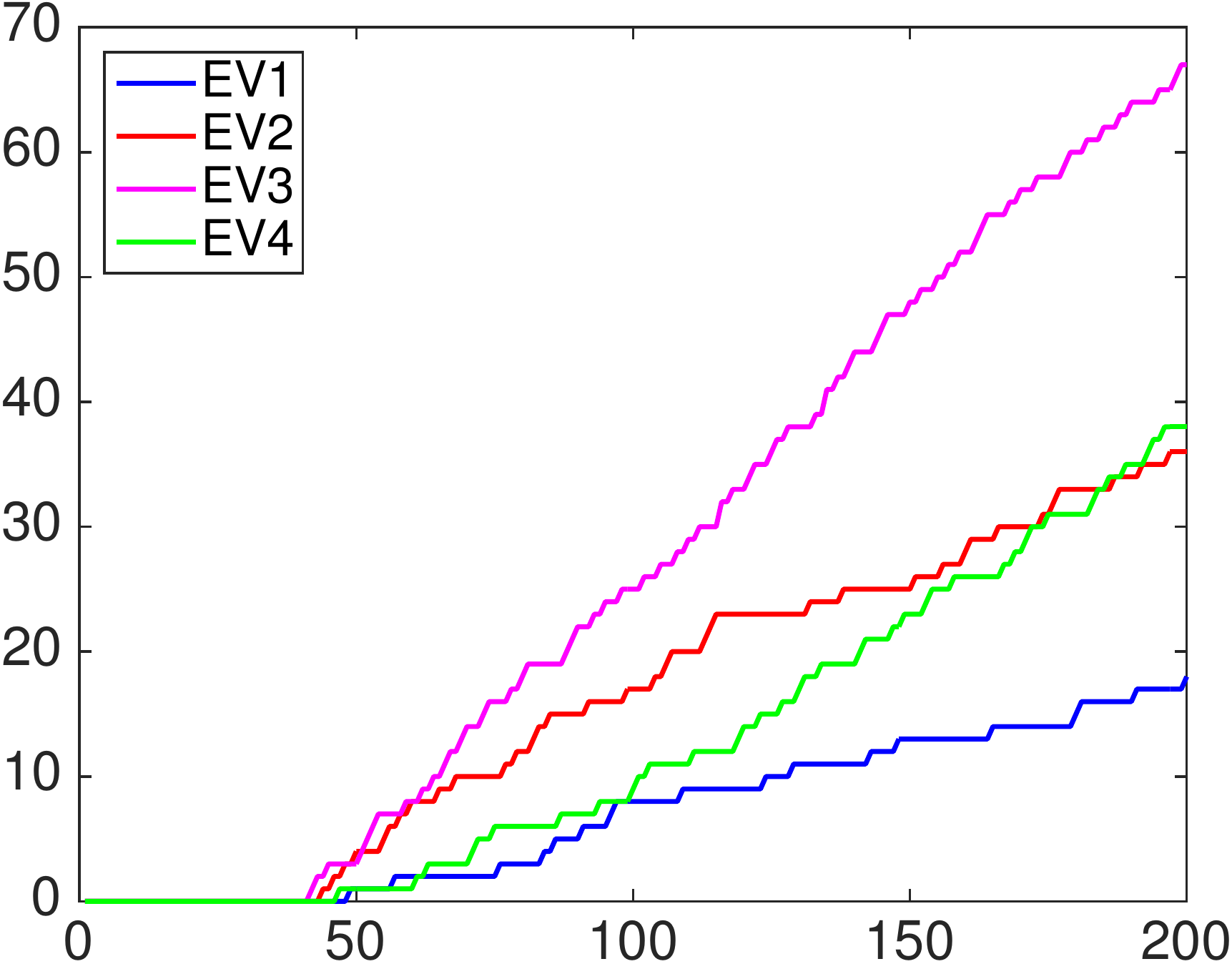}
\hspace*{.25mm}
\includegraphics[width=.322\textwidth]{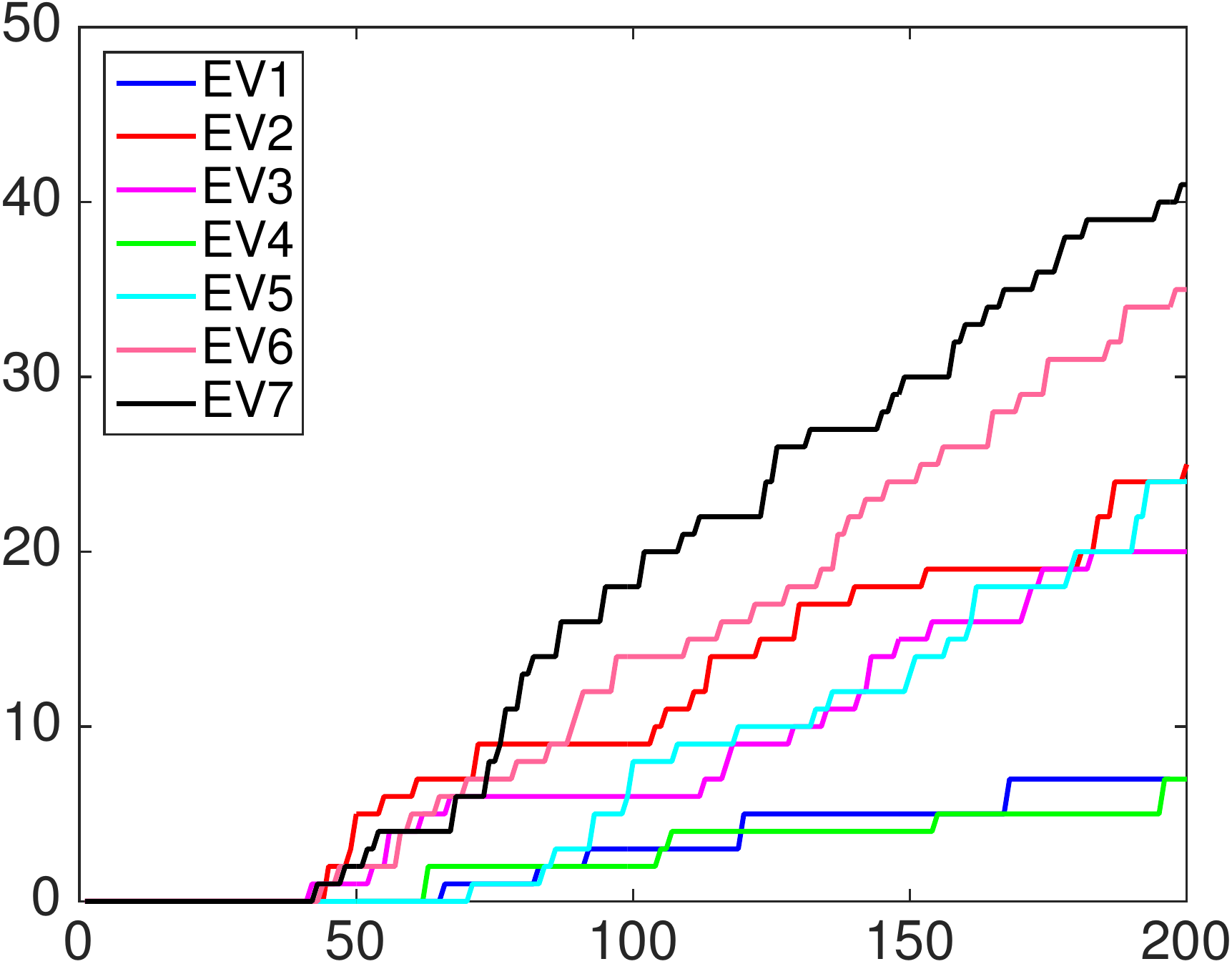}
\hspace*{.25mm}
\includegraphics[width=.33\textwidth]{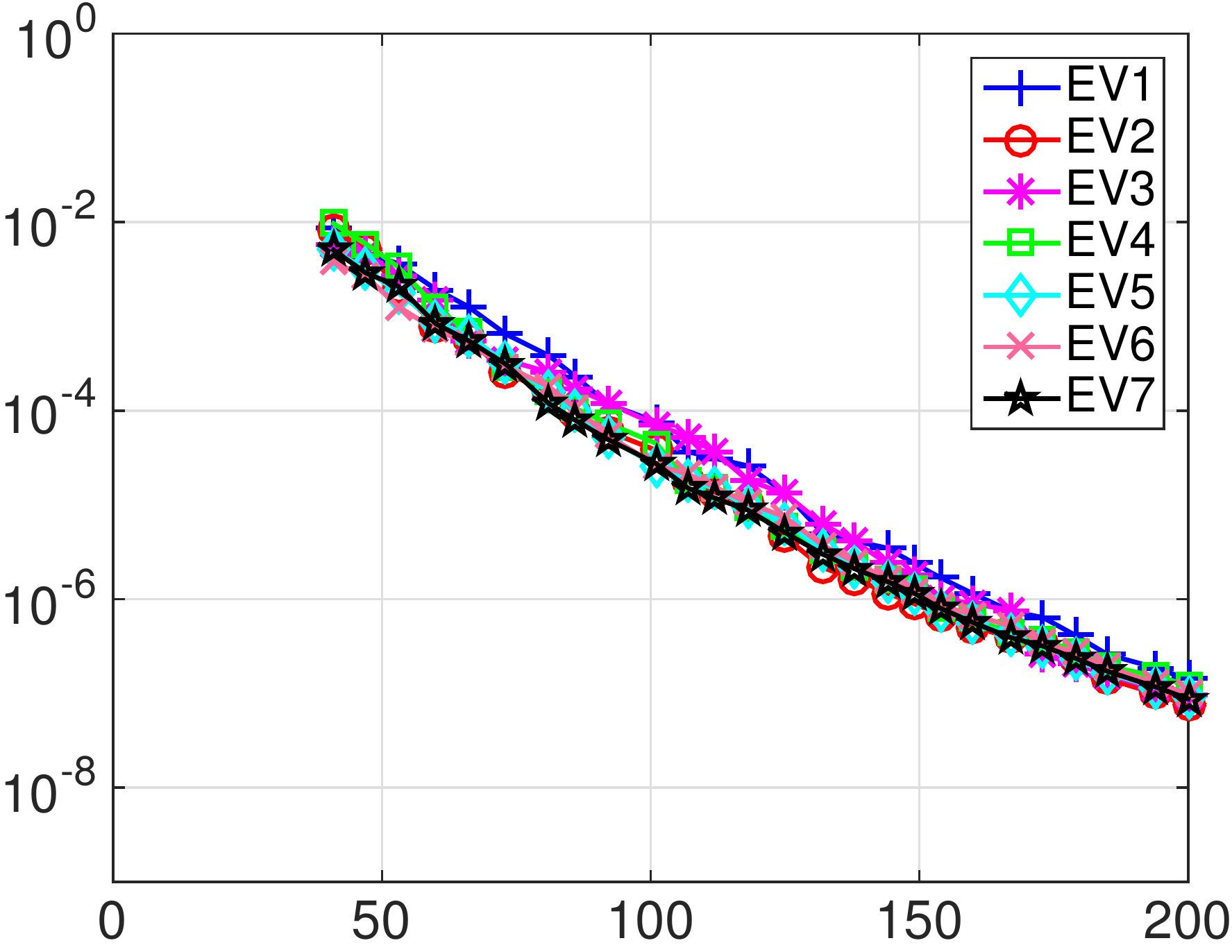}
\caption{
Left and center:
accumulated numbers of chosen eigenfunctions
over the course of Alg.~\ref{alg:greedy_einzeln}
for $\anzEW=4$ and $\anzEW=7$.
Right: error decay for $\anzEW=7$
}
\label{fig:alle-indexcounts}
\end{figure}

To further illustrate the behavior of the single-choice greedy method,
in Fig.~\ref{fig:alle-indexcounts},
we report the accumulated numbers of chosen eigenfunctions
corresponding to $\lambda_1,\ldots,\lambda_\anzEW$
over the course of Alg.~\ref{alg:greedy_einzeln}
for $\anzEW=4$ (left) and $\anzEW=7$ (center),
as selected by the error estimators in line~6.
The reason for the greedy algorithm not selecting any eigenfunctions before a basis size of $40$ is that this is the size of our initial space. 
The respective error decay for $\anzEW=7$ is depicted in 
Fig.~\ref{fig:alle-indexcounts} (right).
Note that the good convergence
(in particular, similar rates for all outputs of interest simultaneously)
is achieved by a rather uneven distribution.
The diagrams indicate that, for both values of $\anzEW$,
larger eigenvalues as well as possibly double eigenvalues are preferred by the algorithm.
This and the fact that, although fewer eigenfunctions are included for the smaller eigenvalues than for the larger ones, but nevertheless the error decay is equal, mean that the eigenfunctions corresponding to larger eigenvalues are effectively used to approximate the ones corresponding to smaller eigenvalues.

\subsection{Effectivity of the greedy method}\label{ssec:effectivity}

In this section,
we investigate the performance of the greedy method
in more detail.
For this purpose,
we also consider the effectivity
 numbers $\gamma_i$, $1 \leq i \leq K$,  of the error estimators and its maximal ratio $R$ defined by
$$
\gamma_i:= \frac{1}{\#\Xi_\mathrm{test}} \sum_{\mu \in \Xi_\mathrm{test}} \frac{\eta_i(\mu) \cdot \lambda_i(\mu)}{\lambdaRB{i}(\mu)-\lambda_i(\mu)},
\quad R := \frac{\max_{i=1, \ldots , K} \gamma_i}{\min_{i=1, \ldots , K} \gamma_i} .
$$
As already mentioned,
the estimators derived in Sect.~\ref{sec:estimator} are of asymptotic character
and therefore generally not reliable for small $N$.
To prevent a misleading selection of basis functions in the first few greedy steps,
the initialization described in Sect.~\ref{ssec:smallPOD} is used
to generate an initial basis.

\begin{figure}[th]
\includegraphics[width=.3\textwidth]{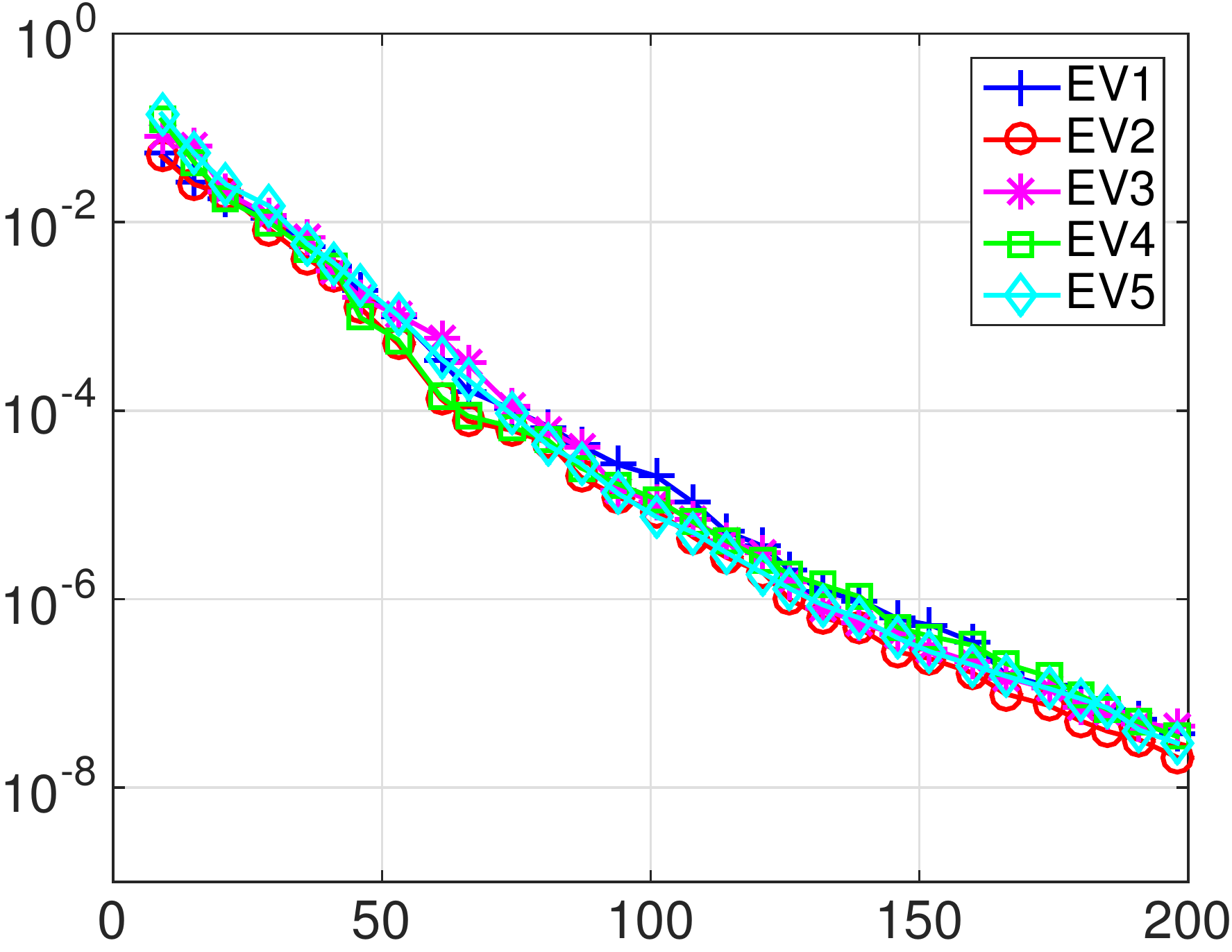}
\includegraphics[width=.3\textwidth]{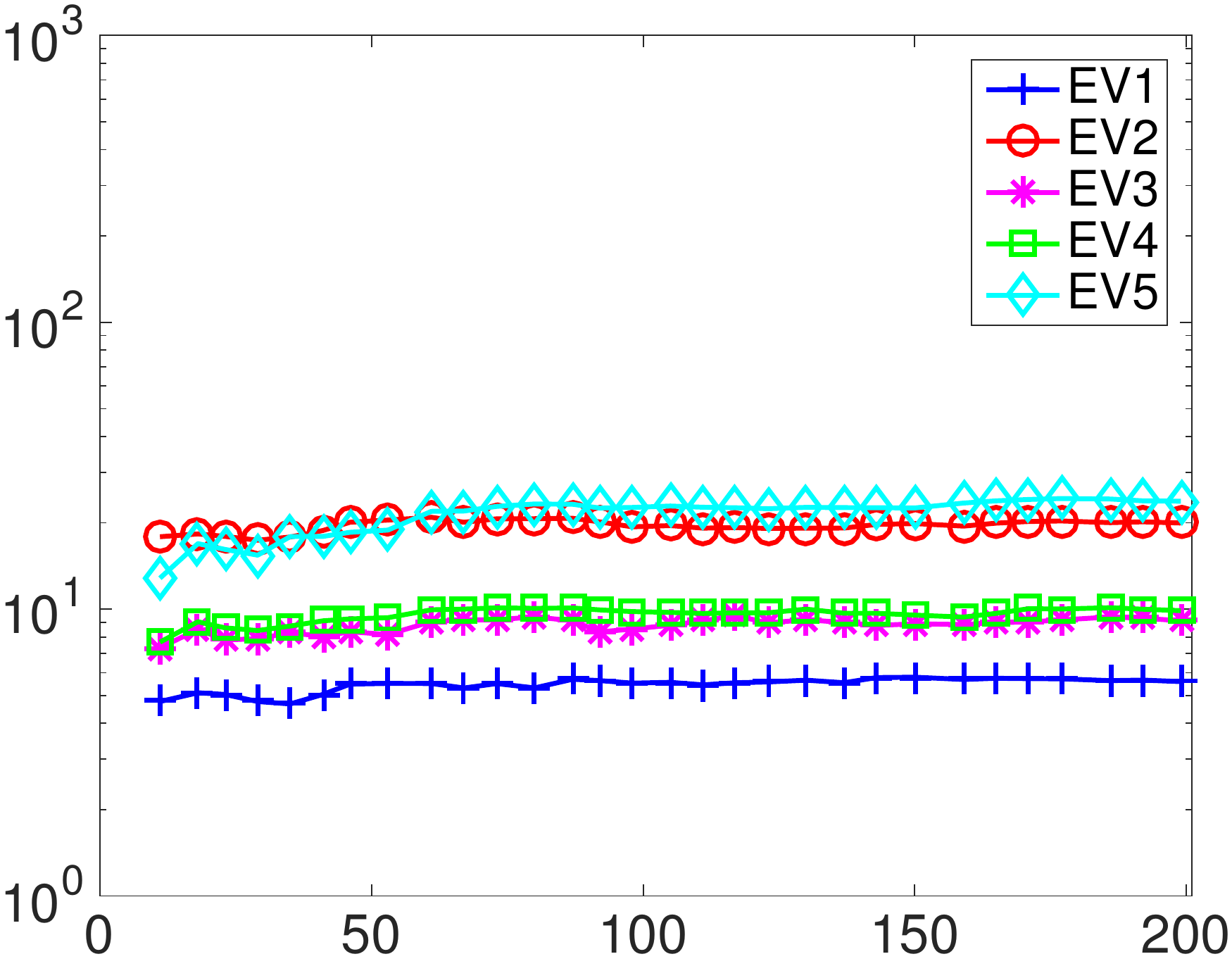}\hspace*{2mm}
\includegraphics[width=.295\textwidth]{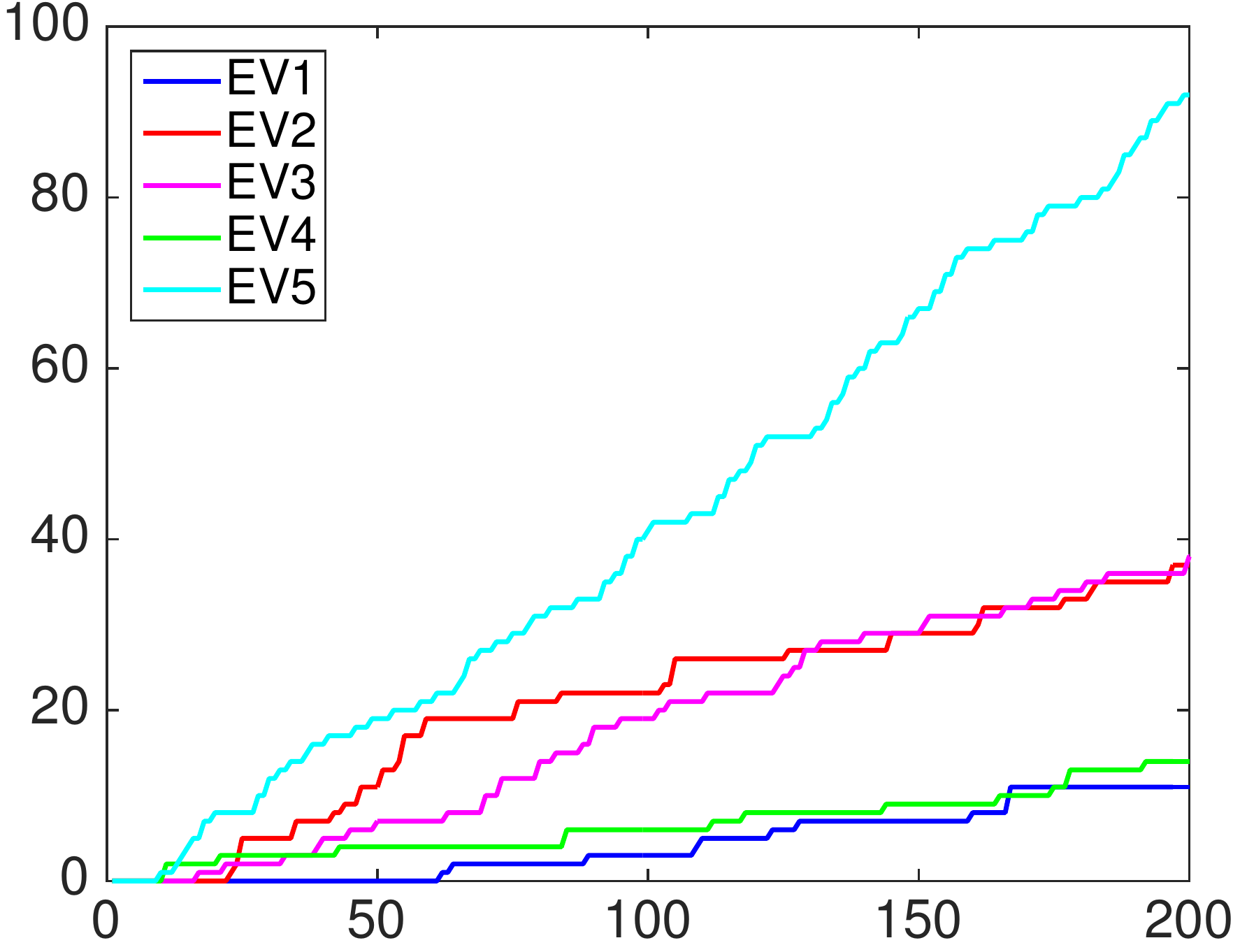} \\[1mm]
\includegraphics[width=.3\textwidth]{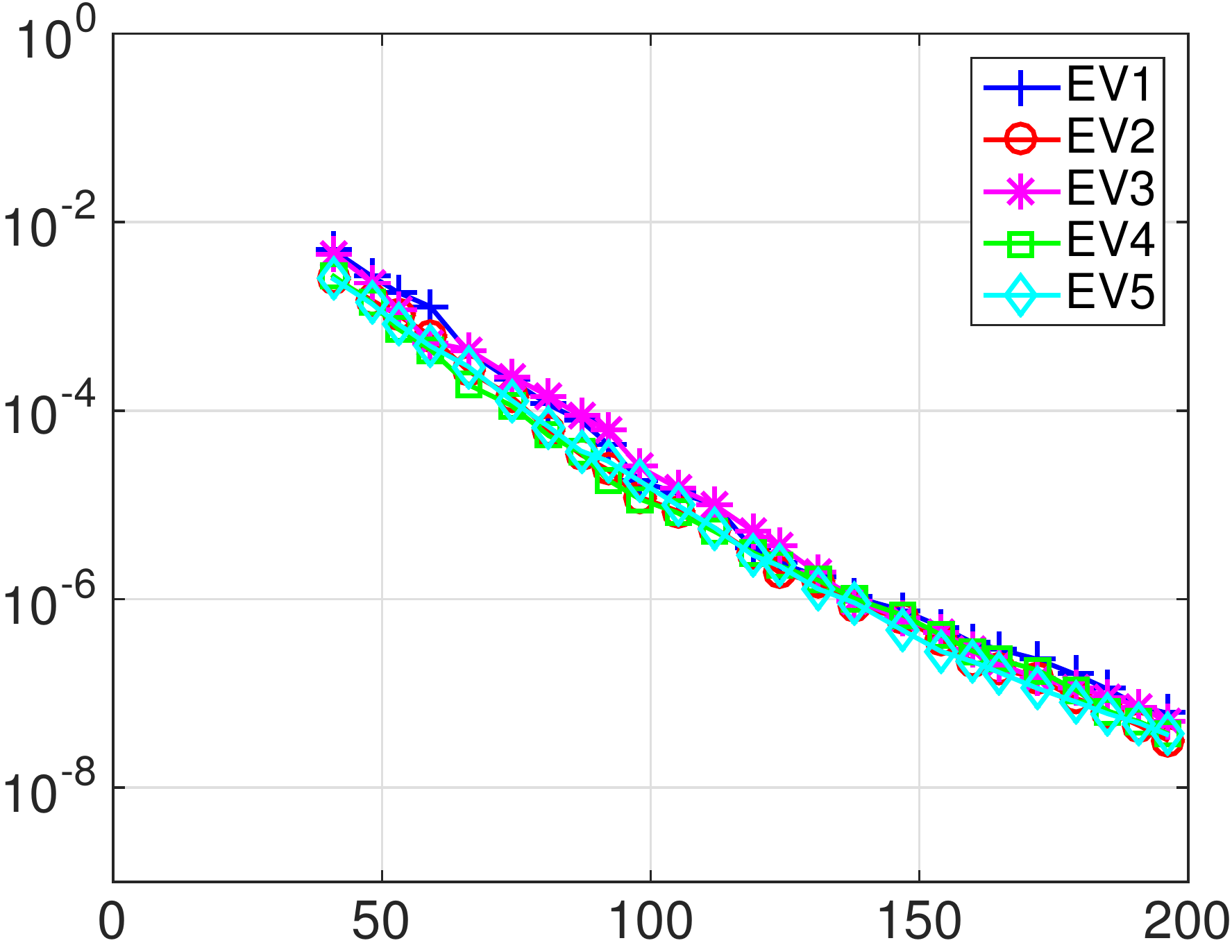}
\includegraphics[width=.3\textwidth]{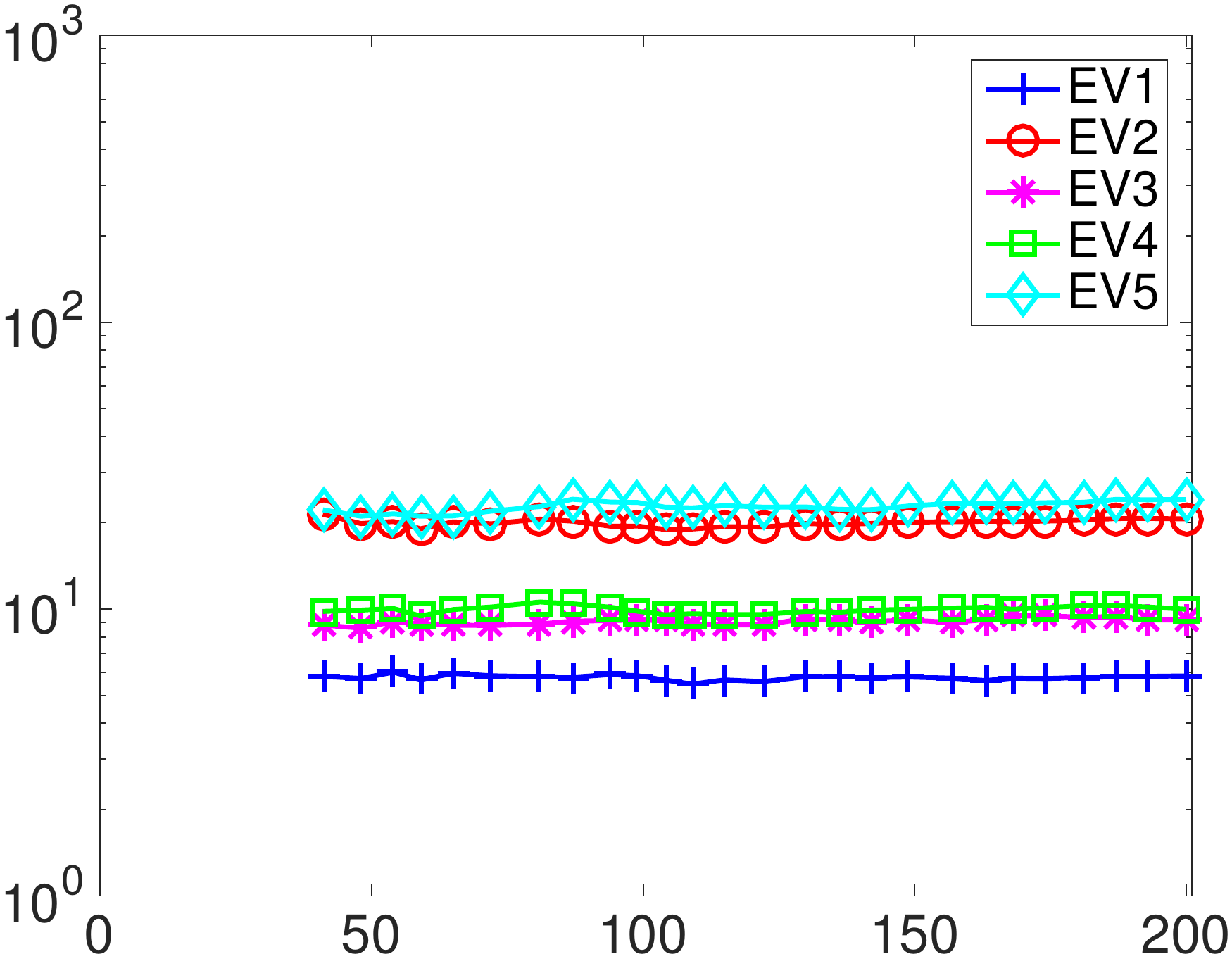}\hspace*{2mm}
\includegraphics[width=.295\textwidth]{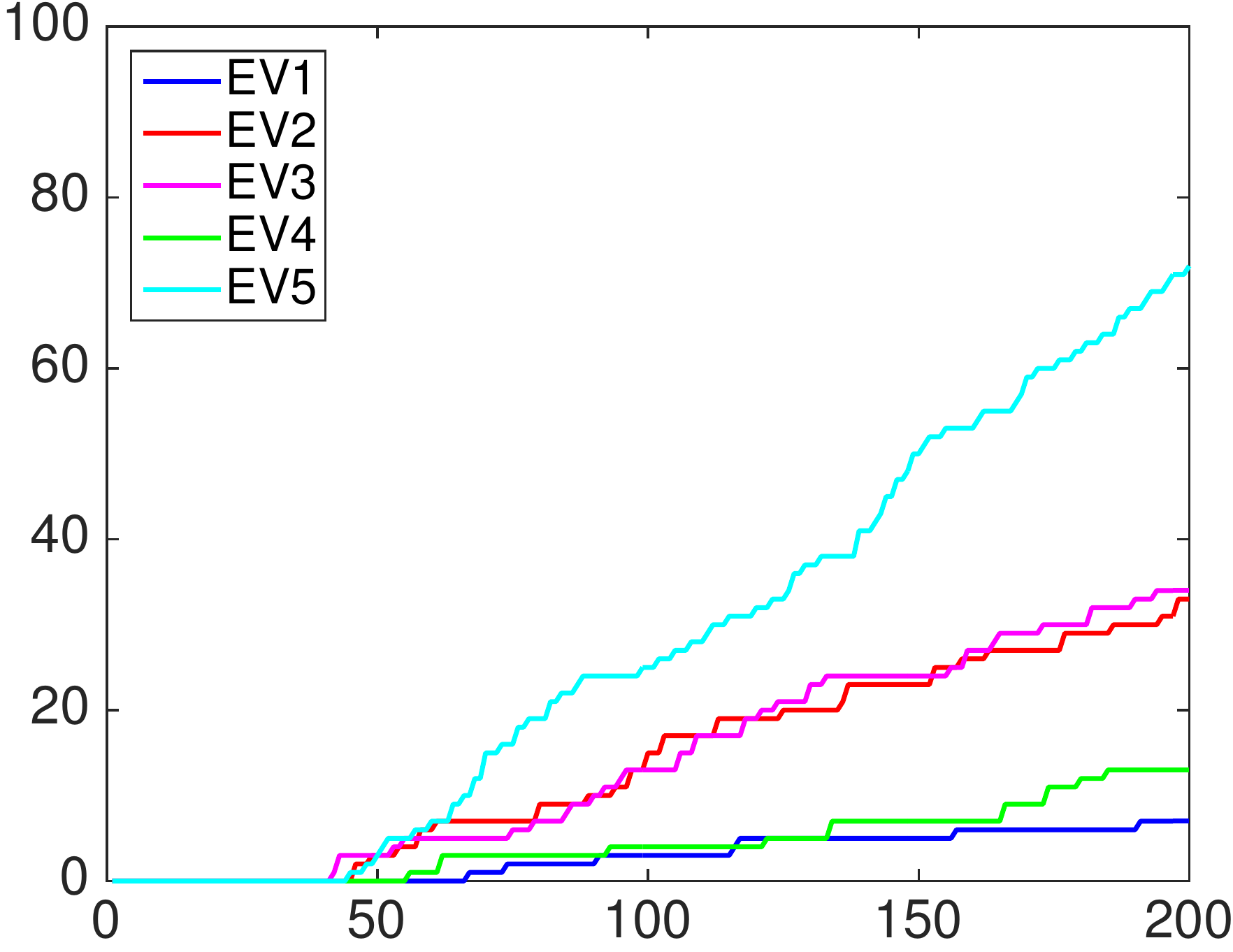}
\caption{Comparison of Alg.~\ref{alg:greedy_einzeln}
without (top) and with (bottom) the initialization described in Sect.~\ref{ssec:smallPOD}
for $\anzEW=5$ { for a selected case in which the greedy algorithm without initialization does not fail.}
RB error decay (left),
effectivity numbers (center)
and accumulated index counts (right)}
\label{fig:POD-and-Greedy-vs-Greedy}
\end{figure}

Fig.~\ref{fig:POD-and-Greedy-vs-Greedy} shows the error decay (left),
the effectivity numbers of the a~posteriori estimators (center)
and the accumulated index counts (right) for $\anzEW=5$
with and without the initialization.
In this case,
a similar convergence is achieved for both algorithms, and 
the index count plots also shows a similar behavior.
{
In the preasymptotic range, we observe a difference in the effectivity numbers. Without initialization these numbers possibly depend sensitively on the selected snapshots. While this does not influence the overall performance for $\anzEW=5$, for $\anzEW=7$ we do get extremely poor results if we start directly with the greedy algorithm. This is caused by the fact that the approximation of $\tilde{d}_i(\mu)$ by $d_i(\mu)$ is then not reliable. Thus we always include the initialization step in our adaptive algorithms.
}

\begin{figure}[th]
\begin{center}
\includegraphics[width=.33\textwidth]{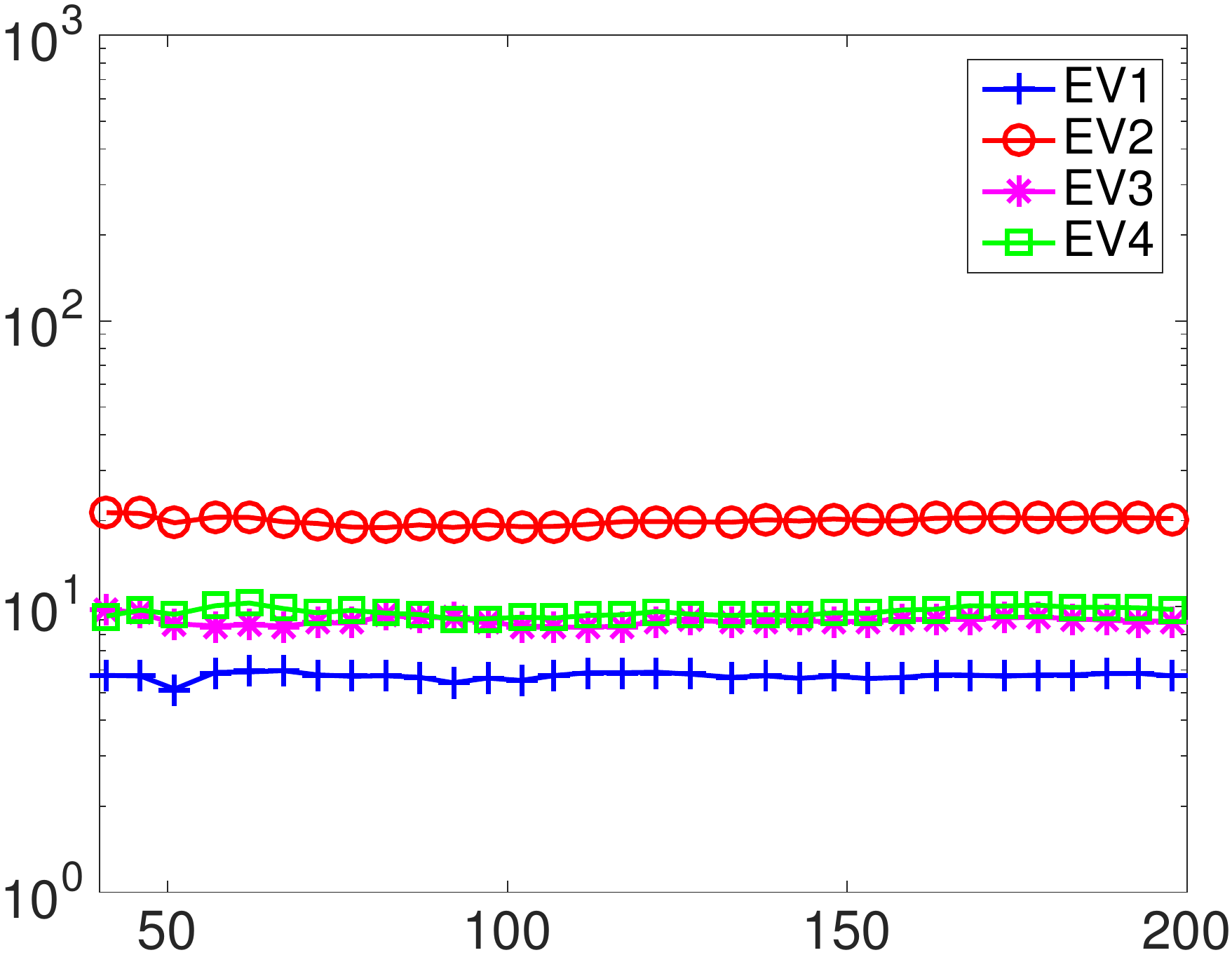}\hspace*{0.2cm}
\includegraphics[width=.33\textwidth]{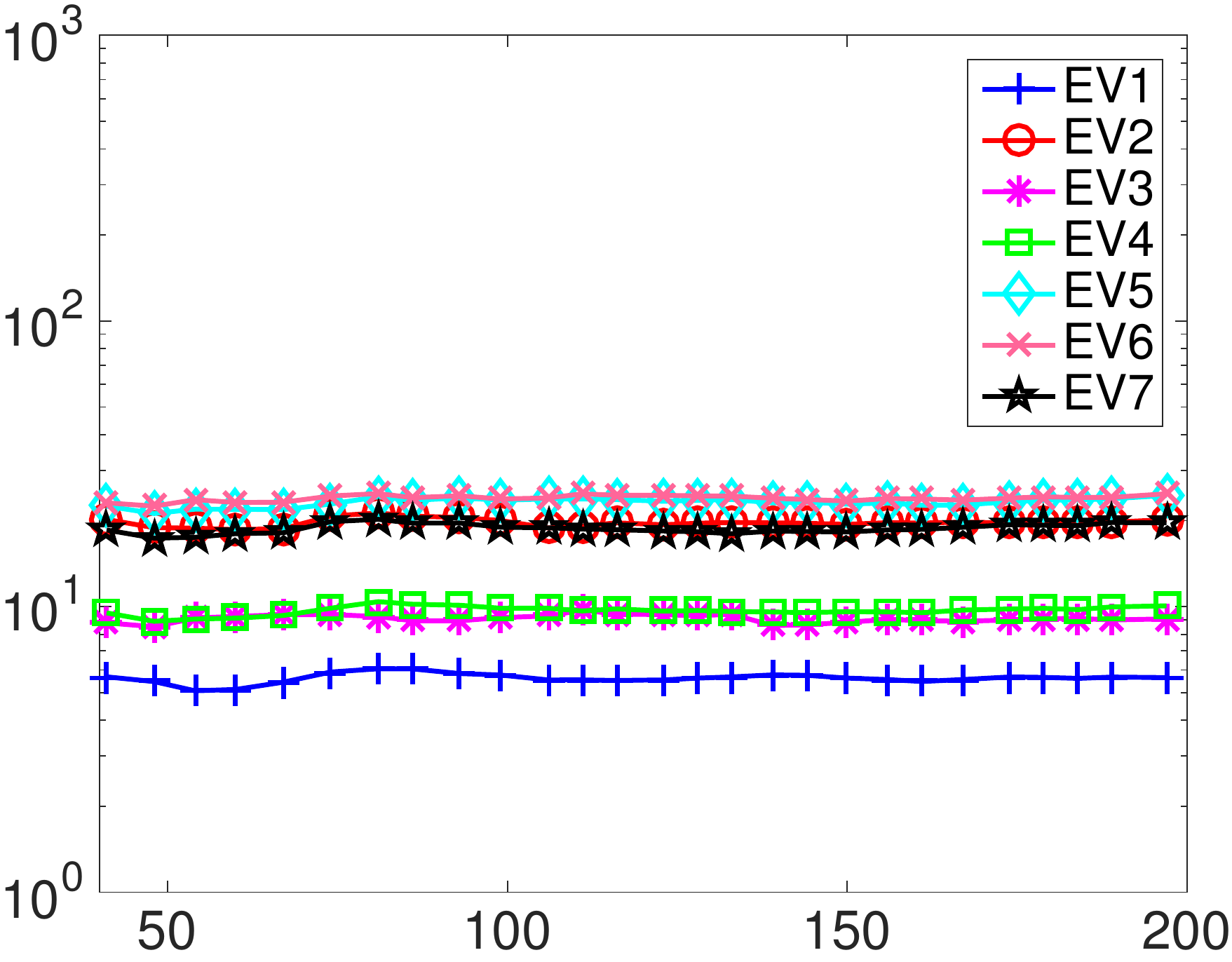}
\end{center}
\caption{Effectivity numbers of the estimators for $\anzEW=4$ (left) and $\anzEW=7$ (right)}
\label{fig:Effectivity}
\end{figure}

In our experiments,
the described initialization always
prevents the effectivity numbers from {having jumps}
and leads to good convergence of the greedy methods.
For instance,
Fig.~\ref{fig:Effectivity} shows the effectivity measures 
corresponding to the error curves
from Fig.~\ref{fig:alle-vs-einzeln} (right)
and Fig.~\ref{fig:alle-indexcounts} (right).
The effectivities are virtually constant and close together which is reflected in a small value of $R$. This is of crucial importance for the performance of our Alg.~\ref{alg:greedy_einzeln}.
In all our settings $R$ is below five, e.\,g., $R=3.41 $  for $\anzEW=4$.
Note that  for $\anzEW=4$  and $\anzEW=7$,
the same eigenvalues show similar effectivities.
A high effectivity ratio $R$ possibly leads to an oversampling of the eigenfunctions associated with the indices of a high effectivity and thus a loss in the performance.
{
At this point although our error estimators are for eigenvalues, we want to show that also the effectivities for the eigenvectors are constant and close together. To do so we depict the results in Fig. \ref{fig:Effectivity_EF}.
}
\begin{figure}[th]
\begin{center}
\includegraphics[width=.33\textwidth]{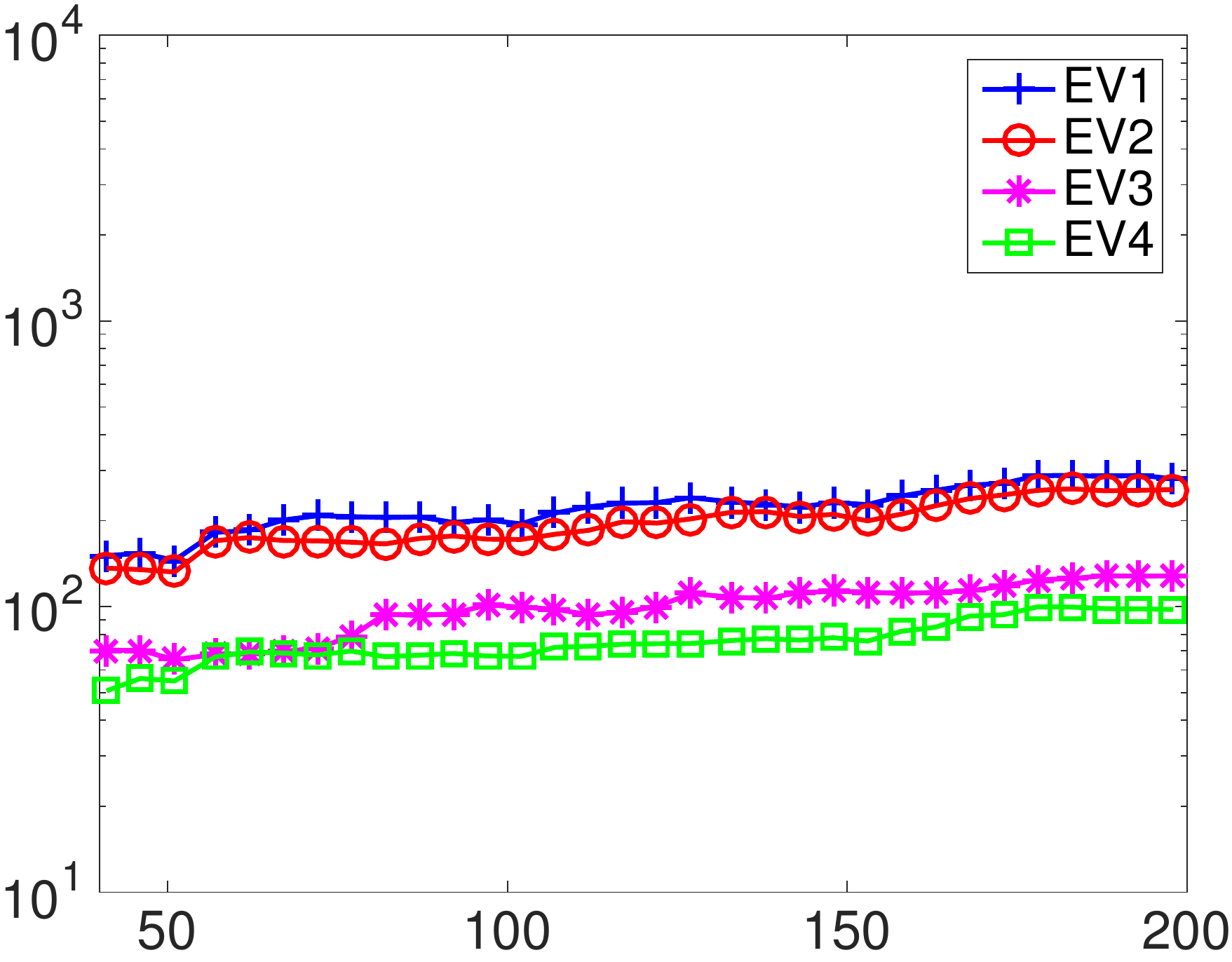}
\end{center}
\caption{Effectivity numbers for the eigenfunctions for $\anzEW=4$}
\label{fig:Effectivity_EF}
\end{figure}

After having demonstrated the performance of the single components of our algorithm,
let us compare the results of our greedy method using the error estimator and the best components
with the convergence of the POD method; cf.\ Sect.~\ref{ssec:POD-RB}.
Comparing the error plots in Fig.~\ref{fig:PODRB_K47} with the ones
in Fig.~\ref{fig:alle-vs-einzeln} (right)
and Fig.~\ref{fig:alle-indexcounts} (right),
we see that we achieve
very similar convergence behavior.
In particular,
the error curves 
of our simultaneous reduced basis approximation
for the individual eigenvalues 
are similarly close to each other.
Moreover,
the accuracy reached at $N\approx 200$ differs only by a factor of roughly ten.
We recall that the POD method uses the full training set (namely 10,000 finite element solutions in this case which leads to a computation time of over 10 h) to reach this accuracy while the greedy method only needs a couple of hundred detailed simulations and the evaluation of the estimator which leads in this case to a compuation time of 6-7 h.  It should be noted that this gap in computation time between POD and Greedy increases further with the complexity of the detailed solution.

{
Let us emphasize that the bounds from \cite{Pau08}, i.\,e., $d_i(\mu)^2$ in the denominator of (\ref{eq:eta_i}) instead of $d_i(\mu)$, lead to a large ratio of the maximal and minimal effectivity value and thus to poorer results in the multiple output case. Highly different effectivity numbers result in an over-selection of eigenfunctions associated with the largest effectivity numbers and thus in a performance loss, hence, to a much less attractive greedy algorithm.
}
\begin{figure}[th]
	\begin{minipage}{0.49\textwidth} 
	\includegraphics[width=\textwidth]{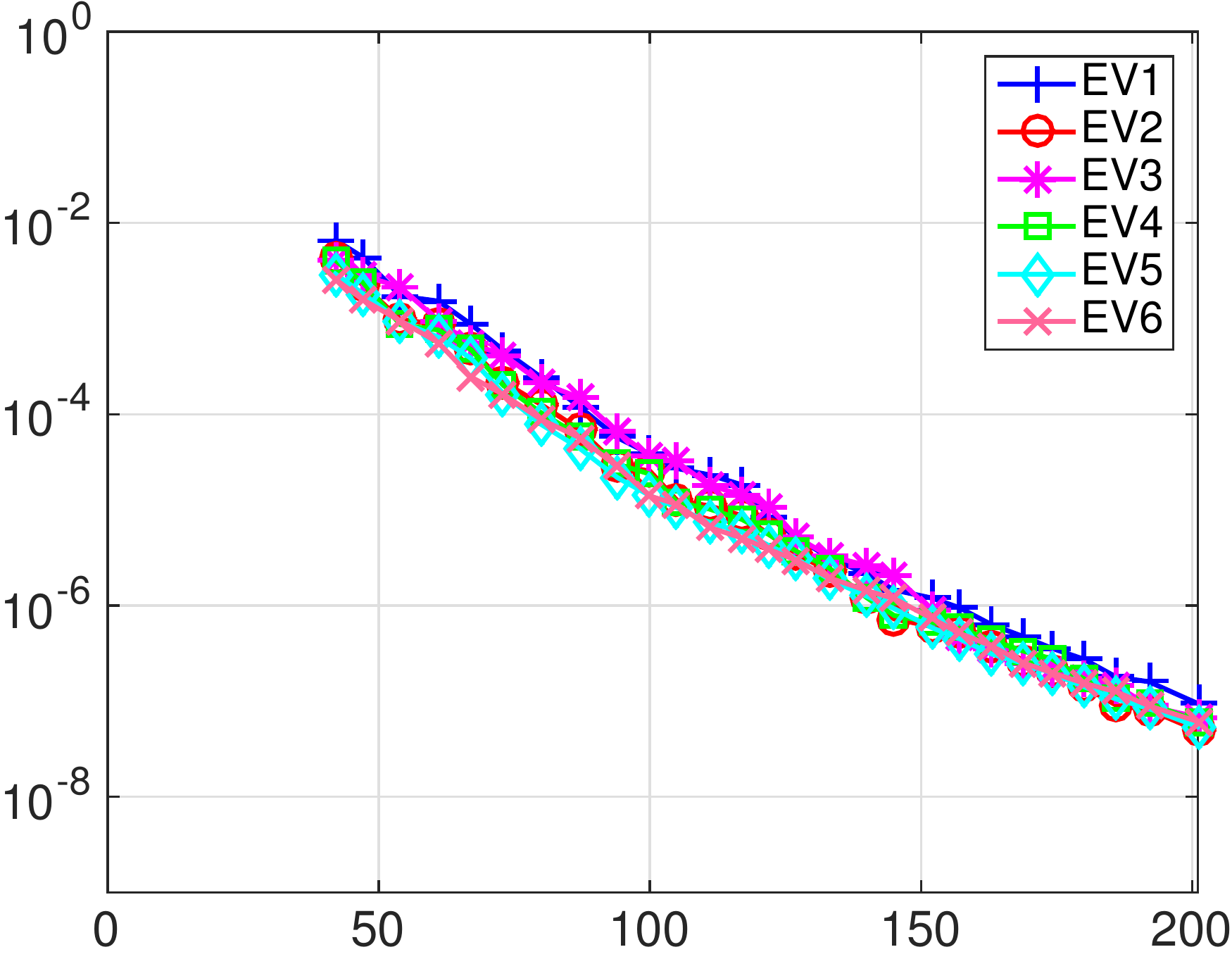}
	\end{minipage}
	\hfill
	\begin{minipage}{0.5\textwidth}
	\includegraphics[width=0.49\textwidth]{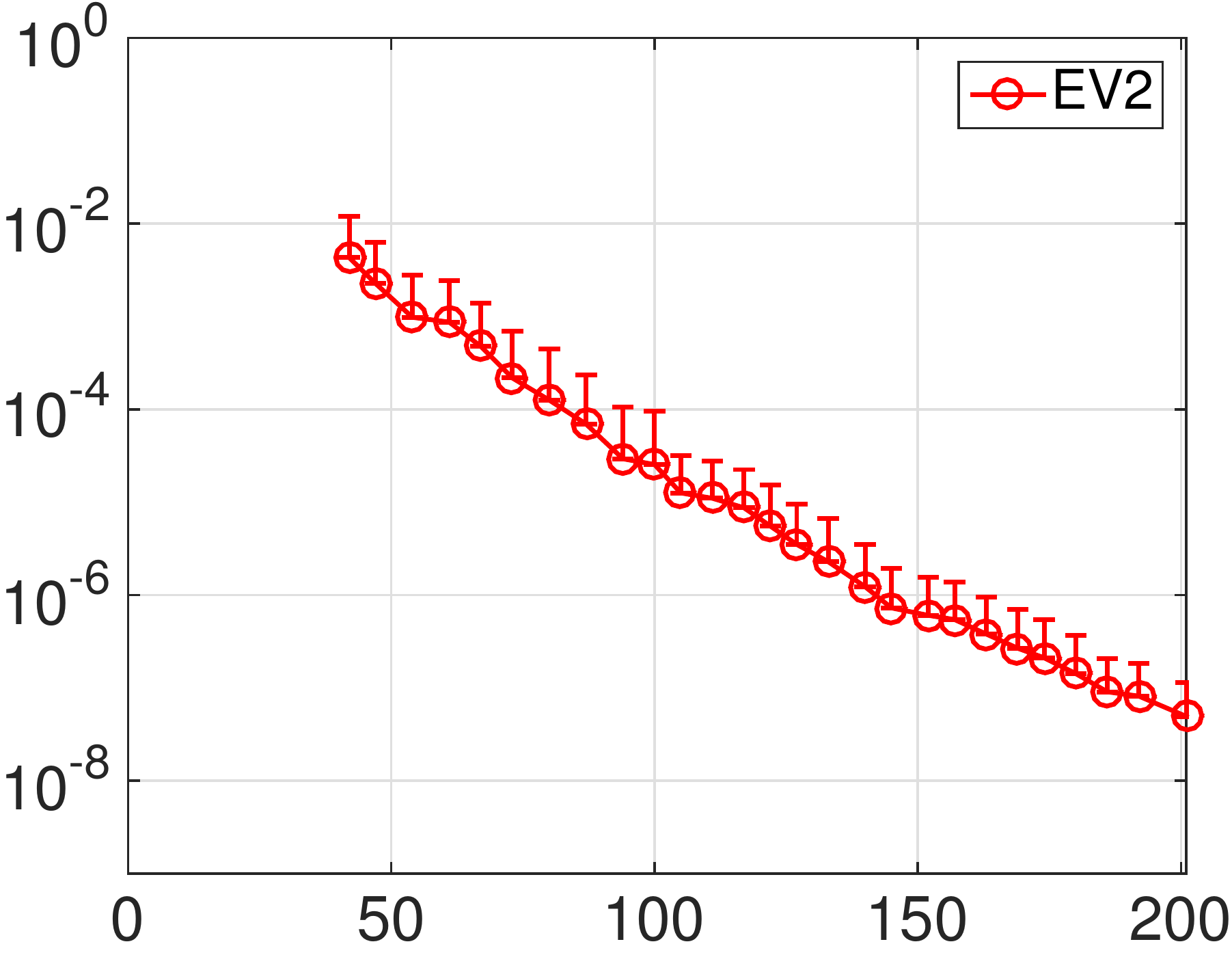}
  	\includegraphics[width=0.49\textwidth]{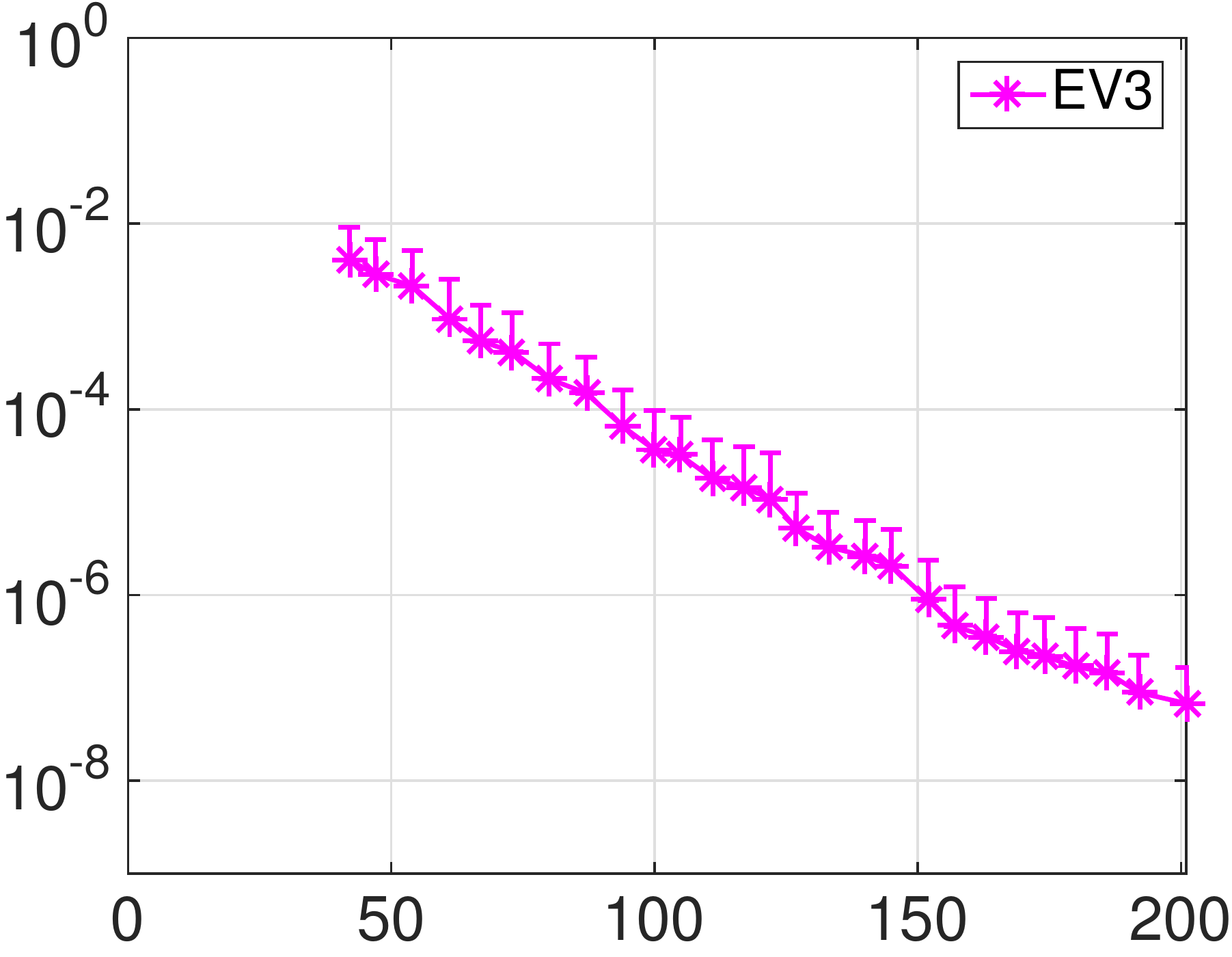} \\
  	\includegraphics[width=0.49\textwidth]{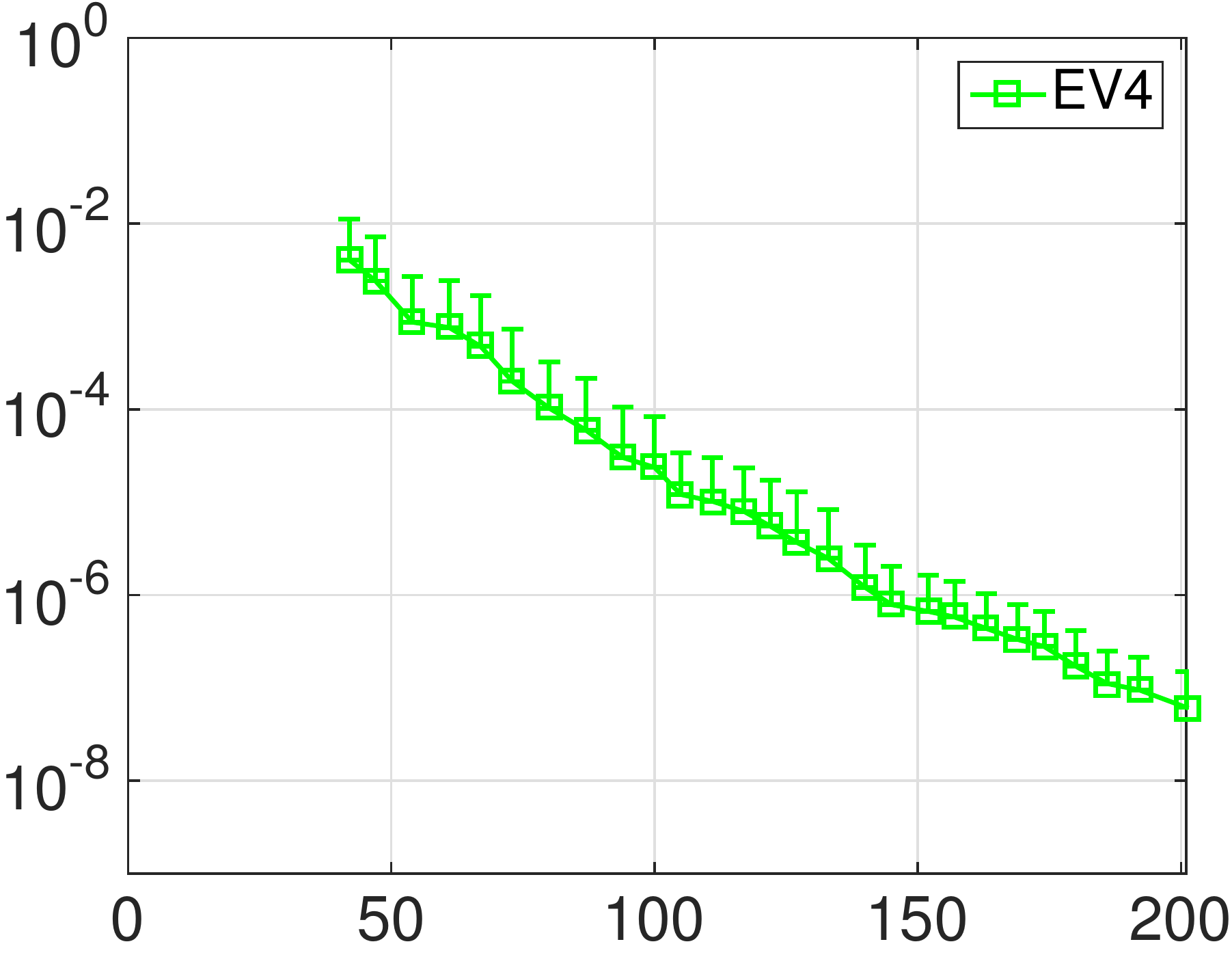} 
	\includegraphics[width=0.49\textwidth]{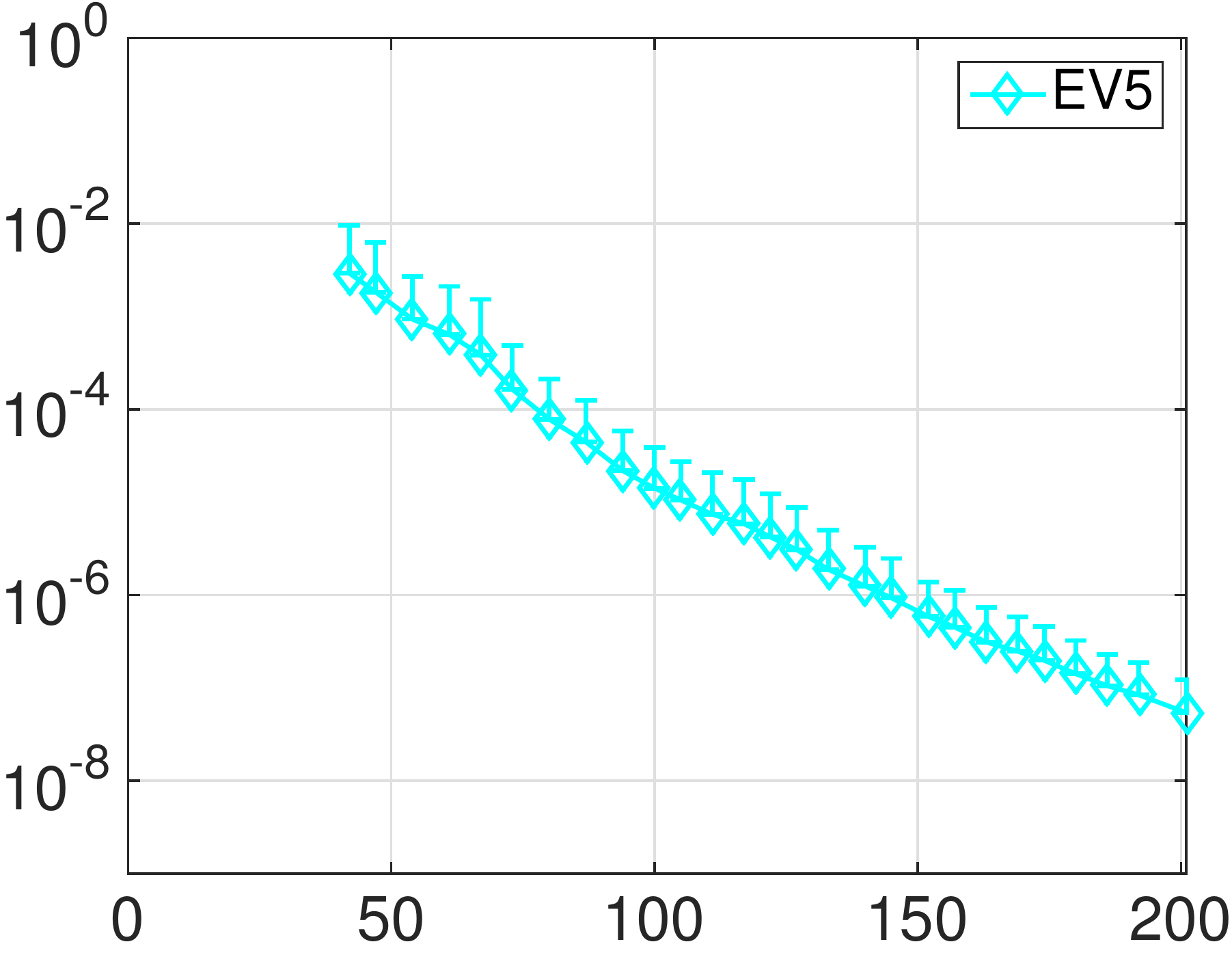}
	\end{minipage}
 \caption{RB discretization error for $\anzEW=6$ with standard deviation (as unidirectional error bar) for EV\,2 to EV\,5}
 \label{fig:STD_EW_6}
\end{figure}

\begin{rmrk}[Error evaluation]\label{rem:statistics}
  For completeness,
  Fig.~\ref{fig:STD_EW_6} shows a convergence plot 
  including the standard deviation
  for $\anzEW=6$.
  In the semilogarithmic plots,
  one can see that the standard deviation is always in the order of the (relative) discretization error itself.
\end{rmrk}

\subsection{Speed-up} \label{ssec:run_time}
The speed-up was calculated serially using MATLAB on a Mac laptop; the standard routine eigs, which is based on ARPACK \cite{ARP98}, was used for solving the eigenvalue problems. We used linear finite elements for the discretization space~$V$.
With our greedy method as introduced above, a significant speed-up in the computation of eigenvalues can be achieved, as is shown for the settings of $K=2$ to $K=7$ eigenvalues in Table \ref{tab:Balken_Speed_up}. 
Here, the calculation of the detailed solutions takes in the range of $3.5$ to $3.6$ seconds, while the calculation of the reduced solution is possible in $0.021$ to $0.078$ seconds, resulting in a speed-up of $140$ to $43$. The higher the value of $N$, the longer the reconstruction time, but in this case the increase is approximately linear in $N$.

Moreover, it should be noted that the more accurately the detailed solution is calculated, the more expensive the detailed calculation becomes while the cost for the calculation of the reduced basis solution will stay in the same range, such that we would achieve even higher speed-ups.
In computations of practical relevance, the detailed and the reduced accuracies have to be adjusted as it is described in \cite{Y015}. Here we are mostly interested in the performance of the RB algorithm, and thus we work with a fixed moderate 
finite element resolution of 15402 DOFs.

\begin{table}[h]\hspace*{-2.5cm}
  \begin{tabular}{|c||c|c|c|c||c|c|c|c||c|c|c|c|}\hline
  $K$& \multicolumn{4}{|c||}{2} &\multicolumn{4}{|c||}{4}&\multicolumn{4}{|c|}{7} \\ \hline
  $N$&50&100&150&200&50&100&150&200&50&100&150&200\\ \hline 
 Detailed solution & \multicolumn{4}{|c||}{3.5} &\multicolumn{4}{|c||}{3.5}&\multicolumn{4}{|c|}{3.6}\\ \hline 
 Reduced solution & 0.025 & 0.028& - & - & 0.037 & 0.043 & 0.050 & 0.055 &  0.057 & 0.067 & 0.075 & 0.082   \\  \hline
 Reconstruction & 0.009 & 0.011 & - & - & 0.009 & 0.012 & 0.016 & 0.019 &  0.010 & 0.013 & 0.018 & 0.021    \\ \hline
 Speed-up & 140 &125 & - & - & 94 & 81 & 70 & 63 & 63 & 53 &48 & 43  \\ \hline
  \end{tabular}
  \vspace{2mm}
  \caption{Timings for the detailed solution and the online calculations
      (reduced solution including error estimation; reconstruction of the finite element solution from the reduced solution) in seconds
      and speed-up numbers
  }
  \label{tab:Balken_Speed_up}
\end{table}

As can be seen in Table \ref{tab:Balken_times}, the computation times for the error estimators ($\eta$) as well as for the required offline components for the error estimators, i.\,e., solution of (\ref{eq:xi_q_n}) and (\ref{eq:xi_0_n}) (``Assembly'') and computation of $\hat{A}$, increase for increasing values of $N$. In the case of $\hat{A}$, the increase is approximately linear. Note that these longer computations will only have to be performed in the offline phase and will not have any impact on the computation times for the online phase. The computation of $g(\mu)$ as defined in (\ref{eq:g_einfach}), which is necessary for the error estimator, takes $0.0042$ seconds. 

\begin{table}[h]
\begin{center}
 \begin{tabular}{|c|c|c|c|} \hline
  $N$ & $\eta$ & $\hat{A}$ & Assembly \\ \hline
    50 & 0.0028 & 1.3672 & 6.076 \\ \hline
  100 & 0.0037 & 2.7640 & 6.156 \\ \hline
  150 & 0.0040 & 4.1144 &  6.272\\ \hline
  200 & 0.0046 & 5.5995 &  6.312\\ \hline
 \end{tabular}
  \vspace{2mm}
   \caption{Timings for single components of the offline phase in seconds}
  \label{tab:Balken_times}
\end{center}
\end{table}

\subsection{Wall-slab configuration}\label{ssec:wallslab}

In this section, we show the ability of the newly developed reduced basis method to approximate multiple eigenvalues in a two-dimensional wall-slab configuration with a thin elastomer layer in between. The domain shape is an L-shape with three non-overlapping subdomains representing the wall, the elastomer and the slab, denoted by $\Omega_1$, $\Omega_2$ and $\Omega_3$, respectively. The corresponding domains are chosen as $\Omega_1=[0,1] \times [0,2.8]$, $\Omega_2=[0,1] \times [2.8,3]$ and $\Omega_3=[0,3] \times [3,4]$. We again used standard linear finite elements with 30702 DOFs for these calculations.

The material parameters $E$ and $\nu$ will again range from $10-100$ and $0.1-0.4$.
Since we aim for large numbers of eigenvalues, we perform our simulations for $\anzEW=20$.
Fig.~\ref{fig:wall_slab} shows that we do not only obtain very good convergence for the eigenvalues (left) but also for the corresponding eigenfunctions (right). 
The error curves chosen to be represented in Fig.~\ref{fig:wall_slab} are representative examples for the eigenvalue and eigenfunction errors in the wall-slab configuration, while the black lines
{
denote the minimum and the maximum of the averaged errors over the $\mu \in \Xi_\mathrm{test}$, respectively.}

\begin{figure}[h]
\begin{center}
\includegraphics[width=.42\textwidth]{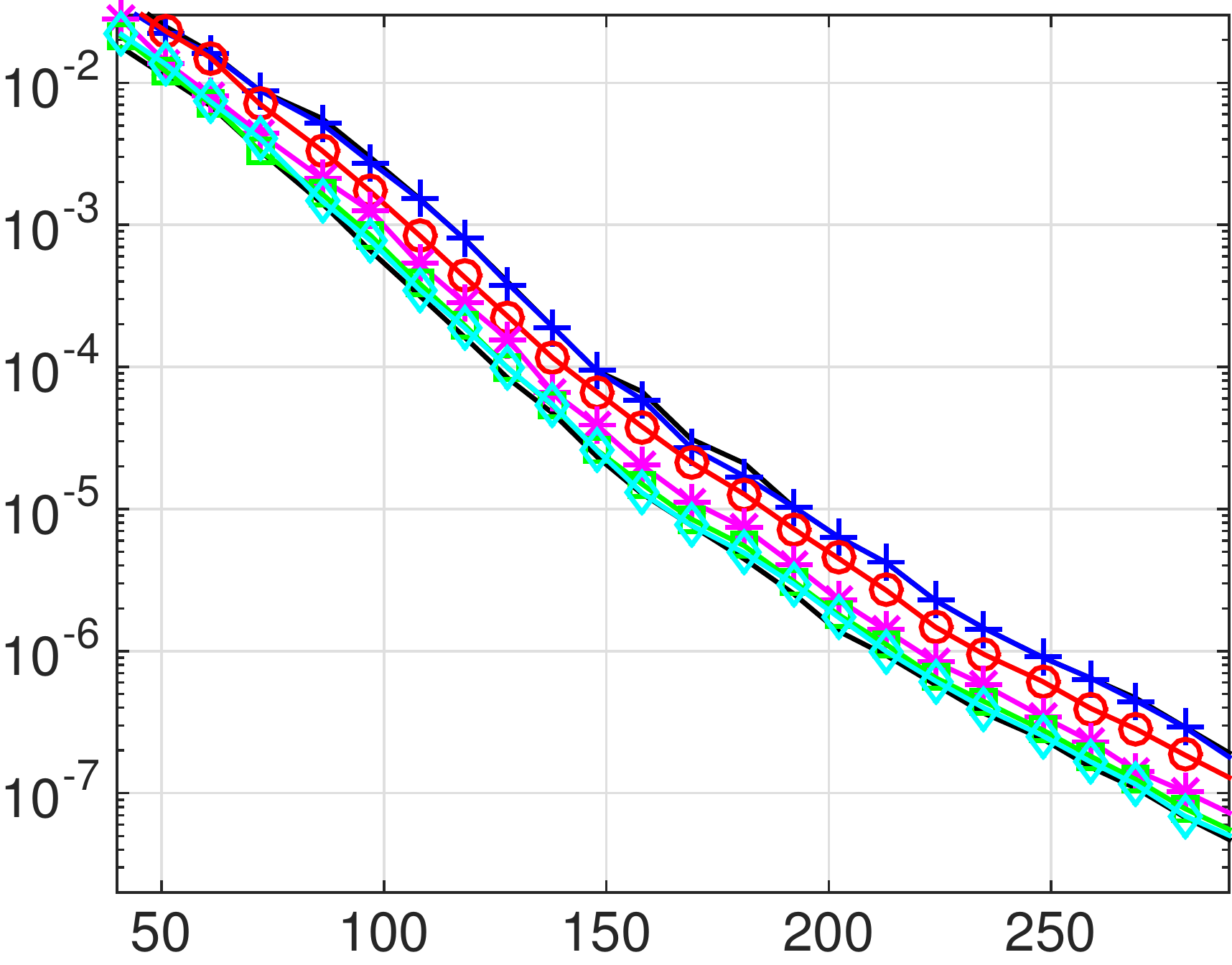}\hspace*{2mm} 
\includegraphics[width=.1\textwidth]{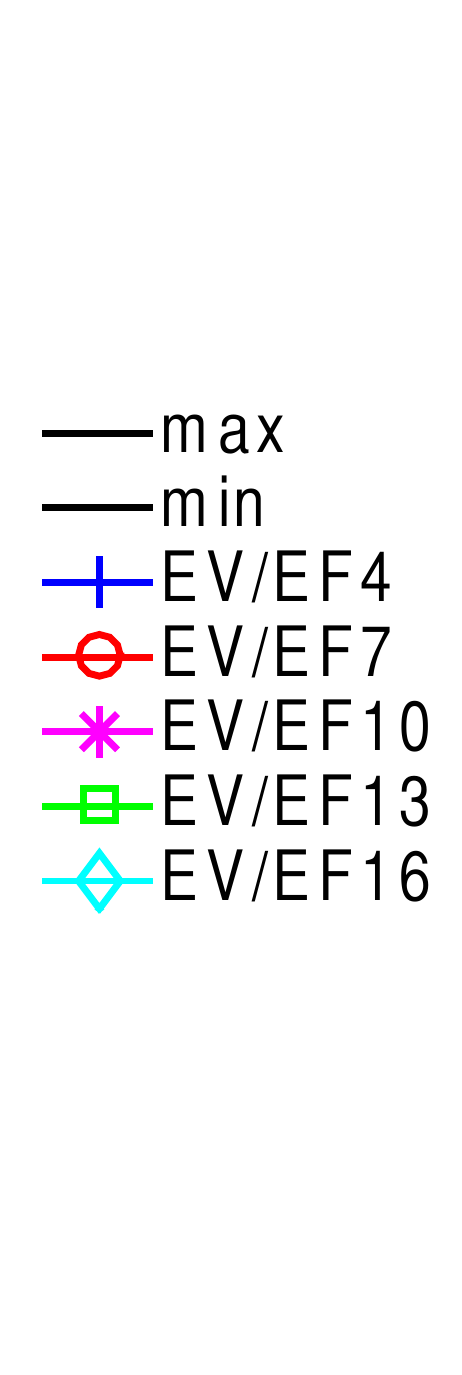} 
\hspace*{2mm}
\includegraphics[width=.42\textwidth]{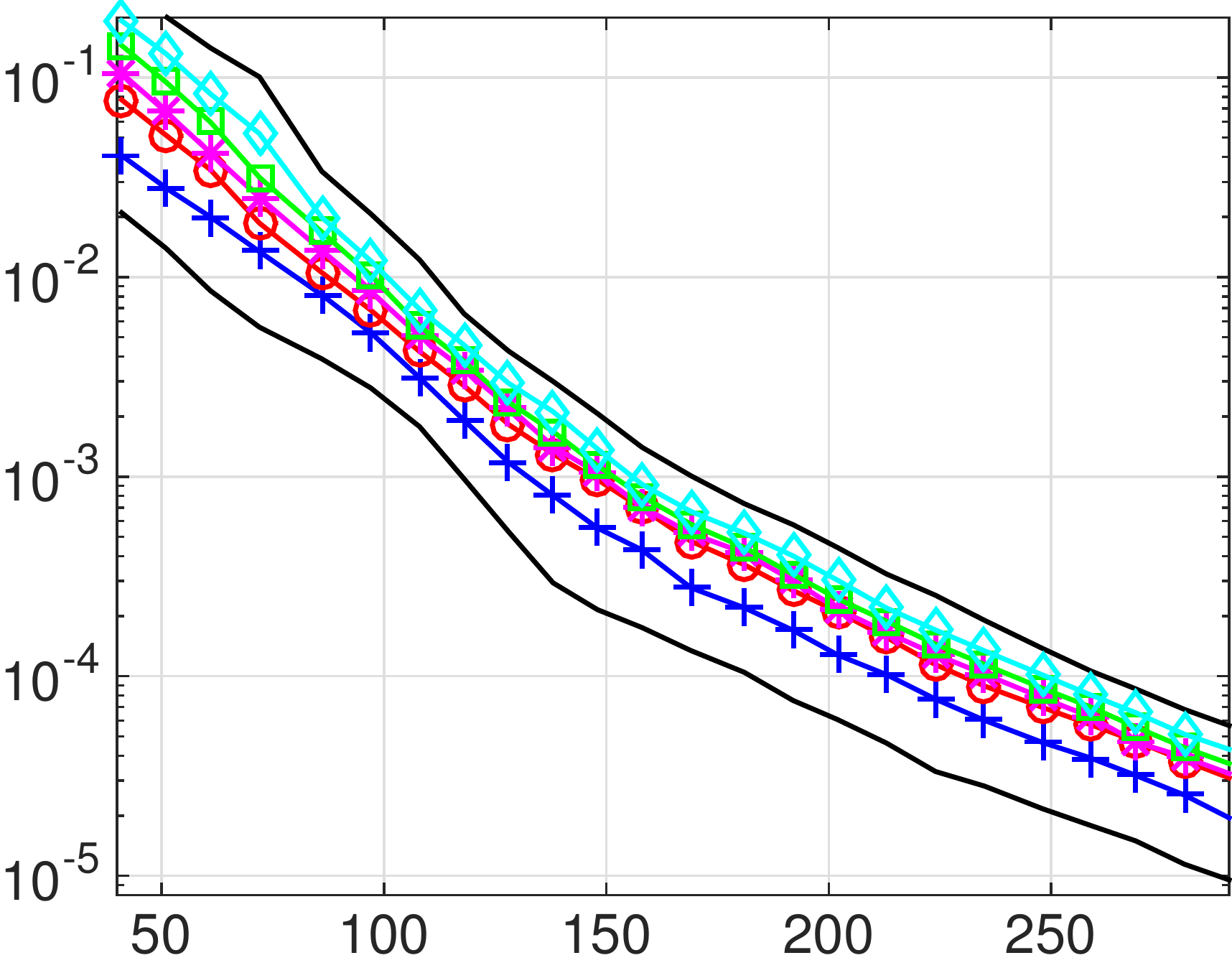}
\end{center}
\caption{Wall-slab configuration with thin elastomer: RB error of eigenvalues (left) and eigenfunctions (right)}
\label{fig:wall_slab}
\end{figure}
The speed-up is similar to the one analyzed in detail in Sect.~\ref{ssec:run_time}.
For the wall-slab configuration, we show in Table \ref{tab:L_speed_up} the computation times in the case of $K=20$ eigenvalues. As can be seen the computation of the detailed solution takes $14.03$ seconds, while the computations of the reduced solutions take between $0.14$ and $0.24$ seconds, depending on the basis size $N$. This results in a speed-up of $100$ for $N=50$ to $58$ if we take $N=300$ for an accuracy of $10^{-7}$.

\begin{table}[h]
\begin{center}
\begin{tabular}{|c||c|c|c|c|c|c|} \hline
$K$&\multicolumn{6}{|c|}{20} \\ \hline
$N$&50&100&150&200&250&300 \\ \hline
Detailed solution&  \multicolumn{6}{|c|}{14.03}\\ \hline
Reduced solution& 0.14& 0.16&0.18& 0.19&0.22&0.24 \\ \hline
Reconstruction& 0.027&0.035&0.041&0.049&0.054&0.065  \\ \hline 
Speed-up& 100&87&78&73&63&58 \\ \hline
\end{tabular}
\vspace{2mm}
\caption{Timings for the detailed solution and the online calculations 
      (reduced solution including error estimation; reconstruction of the finite element solution from the reduced solution)
       for a slab-wall configuration in seconds}
\label{tab:L_speed_up}
\end{center}
\end{table}

\subsection{Three-dimensional example: First floor building}\label{ssec:house}

{
Since we aim to apply our results to the modal analysis for vibro-accoustics of laminated timber structures, as they occur in modern timber buildings, we test the performance of our method on a three-dimensional geometry representing the first floor of a building. 
Although wooden structures consist of orthotropic materials, we will use isotropic material parameters for the ease of computation.

Usually different materials are used in the construction of a building. In this case, we have three different materials for the walls. More precisely, we assume that the outer walls are subdomain one, which consists of one material and that the interior walls can be divided into two more subdomains, namely ordinary walls and load-bearing walls.
Fig.~\ref{fig:house_geo} depicts our geometry and the corresponding domains.
The material parameters $E$ and $\nu$ range from $100-1000$ and $0.1-0.4$.
We perform our simulations for $\anzEW=10$ and use standard finite elements with $20994$ degrees of freedom.

}
\begin{figure}[h]
\begin{center}
\includegraphics[width=.3\textwidth]{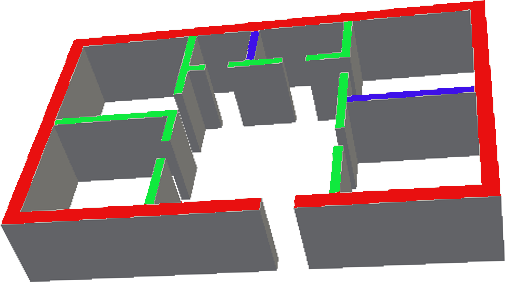}
\end{center}
\caption{Geometry and subdomains: outer walls in red, inner walls load-bearing in green, inner ordinary walls in blue}
\label{fig:house_geo}
\end{figure}

{
The first row in Fig.~\ref{fig:house_eigenvectors} represents the first eigenfunctions for three different parameter sets while the second row represents the corresponding fourth eigenfunctions. We used the parameter sets $\mu^1=(200,0.1,800,0.3,400,0.2)$, $\mu^2=(650,0.36,150,0.25,900,0.11)$ and $\mu^3=(800,0.3,500,0.1,200,0.4)$. It can be observed that the eigenfunctions change significantly depending on the parameters while still being approximated very well by our method.
Fig.~\ref{fig:house_conv} shows the error decay for the $k$-th eigenvalues (left) and eigenfunctions (right), $k\in \{1,3,5,7,9 \}$, as well as the minimum and maximum averaged errors. We again obtain very good convergence.
}

\begin{figure}[h]
\begin{center}
\includegraphics[width=.3\textwidth]{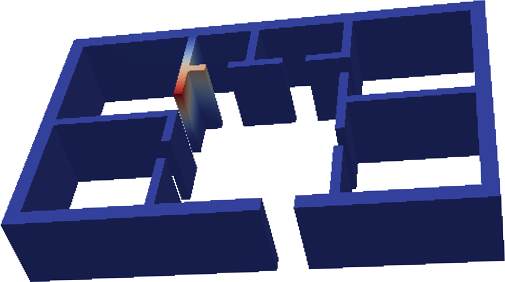}
\includegraphics[width=.3\textwidth]{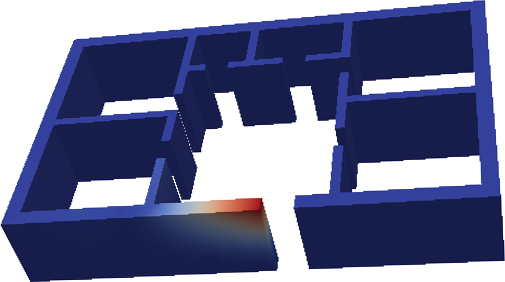}
\includegraphics[width=.3\textwidth]{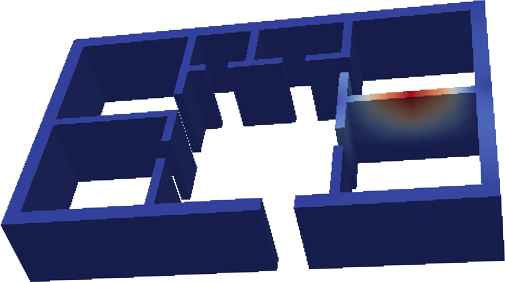}\\
\includegraphics[width=.3\textwidth]{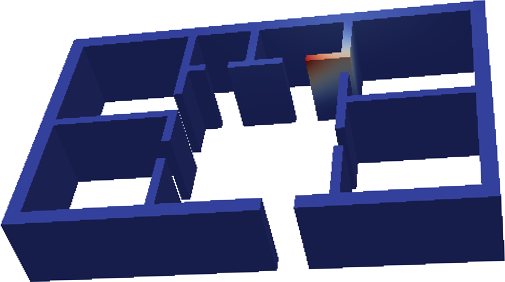}
\includegraphics[width=.3\textwidth]{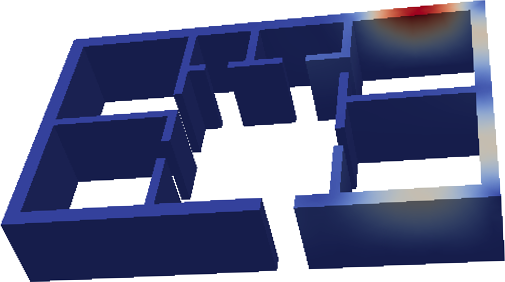}
\includegraphics[width=.3\textwidth]{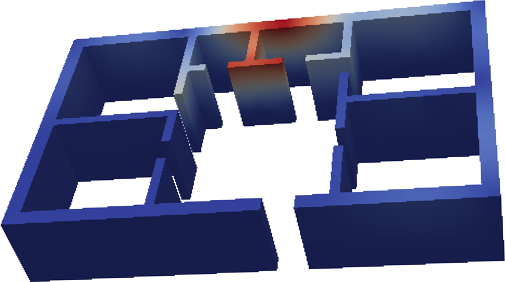}
\end{center}
\caption{Behaviour of the eigenfunctions depending on parameter variations. Top row depicts the first eigenfunction and bottom row the fourth eigenfunction.}
\label{fig:house_eigenvectors}
\end{figure}

\begin{figure}[h]
\begin{center}
\includegraphics[width=.42\textwidth]{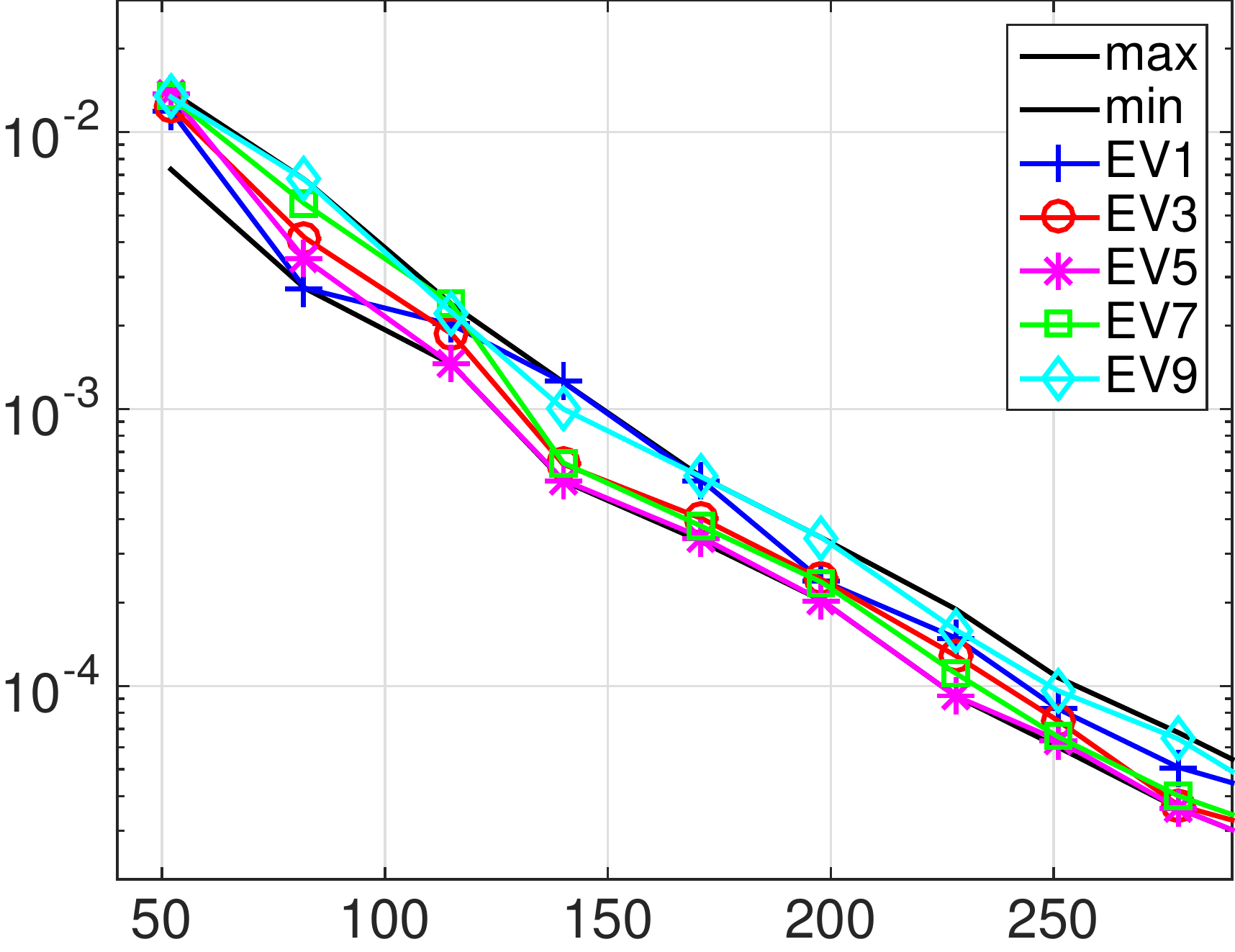}\hspace*{2mm} 
\includegraphics[width=.42\textwidth]{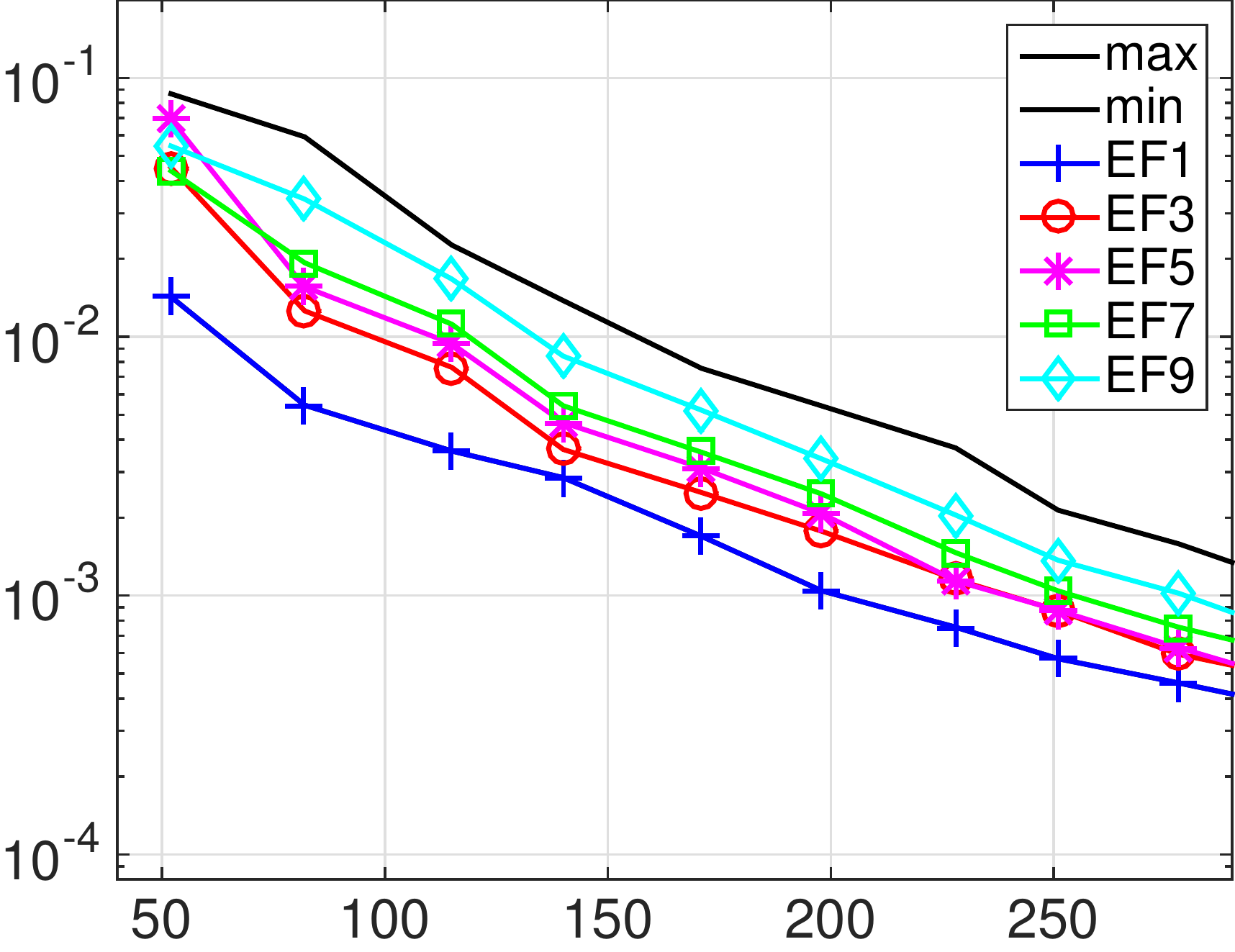}
\end{center}
\caption{First floor configuration: RB error of eigenvalues (left) and eigenfunctions (right)}
\label{fig:house_conv}
\end{figure}

{
The speed-up in the three-dimensional setting is even more significant.
For the first floor of the building, we show in Table \ref{tab:house_speed_up} the computation times for $K=10$ eigenvalues. We observe that the computation of the detailed solution takes $31.59$ seconds, while the computations of the reduced solutions take between $0.084$ and $0.142$ seconds, depending on the basis size $N$. This results in a speed-up of $376$ for $N=50$ to $222$ if we take $N=300$ for an accuracy of $10^{-5}$.
}

\begin{table}[h]
\begin{center}
\begin{tabular}{|c||c|c|c|c|c|c|} \hline
$K$&\multicolumn{6}{|c|}{10} \\ \hline
$N$&50&100&150&200&250&300 \\ \hline
Detailed solution&  \multicolumn{6}{|c|}{ 31.59 }\\ \hline
Reduced solution& 0.084 & 0.096 &0.102& 0.111 &  0.125 & 0.142\\ \hline
Reconstruction& 0.021 &   0.027&0.031& 0.035 &  0.040 &  0.046\\ \hline 
Speed-up&          376   &  329  & 309  & 284 &  252 & 222\\ \hline
\end{tabular}
\vspace{2mm}
\caption{Timings for the detailed solution and the online calculations for the first floor in seconds}
\label{tab:house_speed_up}
\end{center}
\end{table}

%% file: conclusion.tex
\section{Conclusion}
\label{sec:conclusion}
In this paper, we have developed a model reduction framework for 
parameterized {elliptic} eigenvalue problems
and applied it numerically to linear elasticity.
We have derived an asymptotically reliable error estimator for eigenvalues, even for higher multiplicities,
that also facilitates an online-offline decomposition.
Several eigenvalues have been approximated simultaneously by
a single reduced space for the variational approximation of the eigenvalue problem.
Altogether we achieve very effective tailored greedy strategies
for the construction of efficient reduced basis spaces
for the simultaneous approximation of many eigenvalues.

%% file: references.tex

\pagebreak[3]

\bibliographystyle{siam}
\nocite{*}
{\small
\bibliography{bibliographyfile_RB_EWP}
}

%% file: RB_EWP.bbl
\begin{thebibliography}{10}

\bibitem{BaGuOs89}
{\sc I.~Babu\v{s}ka, B.~Guo, and J.~Osborn}, {\em Regularity and numerical
  solution of eigenvalue problems with piecewise analytic data}, SIAM J. Numer.
  Anal., 26 (1989), pp.~1534--1560.

\bibitem{BaOs89}
{\sc I.~Babu\v{s}ka and J.~Osborn}, {\em Finite {E}lement-{G}alerkin
  approximation of the eigenvalues and eigenvectors of selfadjoint problems},
  Math. Comput., 52 (1989), pp.~275--297.

\bibitem{BaOs91}
\leavevmode\vrule height 2pt depth -1.6pt width 23pt, {\em Eigenvalue
  problems}, in Handbook of Numerical Analysis, P.~Ciarlet and J.~Lions, eds.,
  vol.~2, 1991, pp.~641--787.

\bibitem{BICoDaDePeWo11}
{\sc P.~Binev, A.~Cohen, W.~Dahmen, R.~DeVore, G.~Petrova, and P.~Wojtaszczyk},
  {\em Convergence rates for greedy algorithms in reduced basis methods}, SIAM
  Journal of Mathematical Analysis, 43 (2011), pp.~1457--1472.

\bibitem{BuMaPaPrTu12}
{\sc A.~{Buffa}, Y.~{Maday}, A.~T. {Patera}, C.~{Prud'homme}, and
  G.~{Turinici}}, {\em {{\it A priori} convergence of the greedy algorithm for
  the parametrized reduced basis method.}}, {ESAIM, Math. Model. Numer. Anal.},
  46 (2012), pp.~595--603.

\bibitem{ChEhLe13}
{\sc E.~Canc\`es, V.~Ehrlacher, and T.~Leli\`evre}, {\em Greedy algorithms for
  high-dimensional eigenvalue problems}, Constructive Approximation, 40 (2014),
  pp.~387--423.

\bibitem{DaPlWe14}
{\sc W.~Dahmen, C.~Plesken, and G.~Welper}, {\em Double greedy algorithms:
  Reduced basis methods for transport dominated problems}, ESAIM:M2AN, 48
  (2014), pp.~623--663.

\bibitem{DePeWo13}
{\sc R.~DeVore, G.~Petrova, and P.~Wojtaszczyk}, {\em Greedy algorithms for
  reduced bases in banach spaces}, Constructive Approximation, 37 (2013),
  pp.~455--466.

\bibitem{DrHaKaOh12}
{\sc M.~Drohmann, B.~Haasdonk, S.~Kaulmann, and M.~Ohlberger}, {\em A software
  framework for reduced basis methods using {DUNE -RB} and {RB}matlab},
  Advances in DUNE,  (2012), pp.~77--88.

\bibitem{DrHaOh11}
{\sc M.~Drohmann, B.~Haasdonk, and M.~Ohlberger}, {\em Adaptive reduced basis
  methods for nonlinear convection--diffusion equations}, Finite Volumes for
  Complex Applications VI Problems \& Perspectives, 4 (2011), pp.~369--377.

\bibitem{ElLi13}
{\sc H.~Elman and Q.~Liao}, {\em Reduced basis collocation methods for partial
  differential equations with random coefficients}, SIAM/ASA J. Uncertainty
  Quantification, 1 (2013), pp.~192--217.

\bibitem{FuMaPaVe15}
{\sc I.~Fumagalli, A.~Manzoni, N.~Parolini, and M.~Verani}, {\em Reduced basis
  approximation and a posteriori error estimates for parametrized elliptic
  eigenvalue problems}, ESAIM: M2AN,  (2016).

\bibitem{GU14}
{\sc S.~Glas and K.~Urban}, {\em On non-coercive variational inequalities},
  SIAM J. Numer. Anal., 52 (2014), p.~2250–2271.

\bibitem{GraeQuSchroeMeWa14}
{\sc N.~Gr\"abner, S.~Quraishi, C.~Schr\"oder, V.~Mehrmann, and U.~{von
  Wagner}}, {\em New numerical methods for the complex eigenvalue analysis of
  disk brake squeal}, in EuroBrake 2014 Conference Proceedings, 2014.

\bibitem{GrPa05}
{\sc M.~Grepl and A.~Patera}, {\em A posteriori error bounds for reduced-basis
  approximations of parametrized parabolic partial differential equations},
  M2AN Math. Model. Numer. Anal., 39 (2005), pp.~157--181.

\bibitem{HaOh08}
{\sc B.~Haasdonk and M.~Ohlberger}, {\em Reduced basis method for finite volume
  approximations of parametrized linear evolution equations}, M2AN Math. Model.
  Numer. Anal., 42 (2008), pp.~277--302.

\bibitem{HaSaWo12}
{\sc B.~Haasdonk, J.~Salomon, and B.~Wohlmuth}, {\em A reduced basis method for
  parametrized variational inequalities}, SIAM Journal of Mathematical
  Analysis, 50 (2012), pp.~2656--2676.

\bibitem{HoKoFriRaWo14}
{\sc T.~Horger, S.~Kollmannsberger, F.~Frischmann, E.~Rank, and B.~Wohlmuth},
  {\em A new mortar formulation for modeling elastomer bedded structures with
  modal-analysis in {3D}}, Adv. Model. Simul. Eng. Sci., 1 (2014).

\bibitem{HuRoSePa07}
{\sc D.~{Huynh}, G.~{Rozza}, S.~{Sen}, and A.~{Patera}}, {\em {A successive
  constraint linear optimization method for lower bounds of parametric
  coercivity and inf-sup stability constants.}}, C. R. Acad. Sci., Paris,
  S\'er. I, 345 (2007), pp.~473--478.

\bibitem{IQRV14}
{\sc L.~Iapichino, A.~Quarteroni, G.~Rozza, and S.~Volkwein}, {\em Reduced
  basis method for the stokes equations in decomposable domains using greedy
  optimization}, ECMI 2014,  (2014), pp.~1 -- 7.

\bibitem{KaVo07}
{\sc M.~Kahlbacher and S.~Volkwein}, {\em Galerkin proper orthogonal
  decomposition methods for parameter dependent elliptic systems}, Discussiones
  Mathematicae: Differential Inclusions, Control and Optimization, 27 (2007),
  pp.~95--117.

\bibitem{ARP98}
{\sc R.~Lehoucq, D.~Sorensen, and C.~Yang}, {\em Arpack users guide: Solution
  of large-scale eigenvalue problems by implicitely restarted arnoldi methods},
  SIAM,  (1998).

\bibitem{LoMaRo06}
{\sc A.~Lovgren, Y.~Maday, and E.~Ronquist}, {\em A reduced basis element
  method for the steady stokes problem}, M2AN Math. Model. Numer. Anal., 40
  (2006), pp.~529--552.

\bibitem{MaMaOlPaRo00}
{\sc L.~Machiels, Y.~Maday, I.~Oliveira, A.~Patera, and D.~Rovas}, {\em Output
  bounds for reduced-basis approximations of symmetric positive definite
  eigenvalue problems}, C. R. Acad. Sci., Paris, S\'er. I, 331 (2000),
  pp.~153--158.

\bibitem{MaPaPe99}
{\sc Y.~Maday, A.~Patera, and J.~Peraire}, {\em A general formulation for a
  posteriori bounds for output functionals of partial differential equations;
  application to the eigenvalue problem}, C. R. Acad. Sci., Paris, S\'er. I,
  328 (1999), pp.~823--828.

\bibitem{MaPaTu002}
{\sc Y.~Maday, A.~Patera, and G.~Turinici}, {\em Global a priori convergence
  theory for reduced basis approximations of single-parameter symmetric
  coercive elliptic partial differential equations}, Comptes Rendus Academie
  des Sciences, Paris, Serie I, Math., 335 (2002), pp.~289--294.

\bibitem{MaPaTu02}
\leavevmode\vrule height 2pt depth -1.6pt width 23pt, {\em A priori convergence
  theory for reduced-basis approximations of single-parametric elliptic partial
  differential equations}, Journal of Scientific Computing, 17 (2002),
  pp.~437--446.

\bibitem{NoPe80}
{\sc A.~Noor and J.~Peters}, {\em Reduced basis technique for nonlinear
  analysis of structures}, AIAA Journal, 18 (1980), pp.~455--462.

\bibitem{PaRo06}
{\sc A.~{Patera} and G.~{Rozza}}, {\em Reduced basis approximation and a
  posteriori error estimation for parametrized partial differential equations}.
\newblock Version 1.0, Copyright MIT 2006--2007.

\bibitem{Pau07a}
{\sc G.~Pau}, {\em Reduced-basis method for band structure calculations}, Phys.
  Rev. E, 76 (2007), p.~046704.

\bibitem{Pau07}
{\sc G.~Pau}, {\em Reduced Basis Method for Quantum Models of Crystalline
  Solids}, PhD thesis, Massachusetts Institute of Technology, June 2007.

\bibitem{Pau08}
{\sc G.~Pau}, {\em Reduced basis method for simulation of nanodevices}, Phys.
  Rev. B, 78 (2008), p.~155425.

\bibitem{Qua14}
{\sc A.~{Quarteroni}}, {\em Numerical Models for Differential Problems}, vol.~8
  of MS\&A, Springer, Milan, 2nd~ed., 2014.

\bibitem{QaMaNe15}
{\sc A.~Quarteroni, A.~Manzoni, and F.~Negri}, {\em Reduced Basis Methods for
  Partial Differential Equations. An Introduction}, Springer, 2015.

\bibitem{RoMaMa06}
{\sc D.~Rovas, L.~Machiels, and Y.~Maday}, {\em Reduced basis output bound
  methods for parabolic problems}, IMA J. Numer. Anal., 26 (2006),
  pp.~423--445.

\bibitem{RoHuPa08}
{\sc G.~{Rozza}, D.~{Huynh}, and A.~{Patera}}, {\em {Reduced basis
  approximation and a posteriori error estimation for affinely parametrized
  elliptic coercive partial differential equations. Application to transport
  and continuum mechanics.}}, {Arch. Comput. Methods Eng.}, 15 (2008),
  pp.~229--275.

\bibitem{RoHuMa13}
{\sc G.~Rozza, D.~B.~P. Huynh, and A.~Manzoni}, {\em Reduced basis
  approximation and a posteriori error estimation for stokes flows in
  parametrized geometries: roles of the inf-sup stability constants},
  Numerische Mathematik, 125 (2013), pp.~115--152.

\bibitem{RoVe07}
{\sc G.~Rozza and K.~Veroy}, {\em On the stability of the reduced basis method
  for stokes equations in parametrized domains}, Comput. Methods. Appl. Mech.
  Engrg., 196 (2007), pp.~1244--1260.

\bibitem{UrAn14}
{\sc K.~Urban and A.~T. Patera}, {\em An improved error bound for reduced basis
  approximation of linear parabolic problems}, Math. Comp, 83 (2014),
  pp.~1599--1615.

\bibitem{UrVoZe14}
{\sc K.~Urban, S.~Volkwein, and O.~Zeeb}, {\em Greedy sampling using nonlinear
  optimization}, in Reduced Order Methods for Modeling and Computational
  Reductions, vol.~9, 2014, pp.~137--157.

\bibitem{UrWi12}
{\sc K.~Urban and B.~Wieland}, {\em Affine decompositions of parametric
  stochastic processes for application within reduced basis methods}, in 7th
  Vienna International Conference on Mathematical Modelling, MATHMOD 2012,
  Vienna, February 15-17 2012.

\bibitem{VaHuKnNgPa14}
{\sc S.~Vallaghe, D.~P. Huynh, D.~J. Knezevic, T.~L. Nguyen, and A.~T. Patera},
  {\em Component-based reduced basis for parametrized symmetric eigenproblems},
  Adv. Model. Simul. Eng. Sci., 2 (2015).

\bibitem{Veroy03}
{\sc K.~Veroy}, {\em Reduced-Basis Methods Applied to Problems in Elasticity:
  Analysis and Applications}, PhD thesis, Massachusetts Institute of
  Technology, June 2003.

\bibitem{Y015}
{\sc M.~Yano}, {\em A minimum-residual mixed reduced basis method: Exact
  residual certification and simultaneous finite-element reduced-basis
  refinement}, ESAIM: M2AN, 50 (2016), pp.~163 -- 185.

\bibitem{ZaVe13}
{\sc L.~Zanon and K.~Veroy-Grepl}, {\em The reduced basis method for an elastic
  buckling problem}, PAMM, Proc. Appl. Math. Mech., 13 (2013), pp.~439--440.

\end{thebibliography}
